\documentclass[11pt,a4paper,reqno, draft]{amsart}
\usepackage{amsmath, amssymb, latexsym, enumerate,  graphicx,tikz, stmaryrd, eufrak,enumitem,MnSymbol}

\usepackage{thmtools}

\usepackage[pdftex]{hyperref}
\usepackage{cleveref}
\usepackage{bbm}
\usepackage{cases}
\usepackage{array}
\usepackage{tikz-cd}
\usepackage[all]{xy}

\usepackage[hmarginratio={1:1},vmarginratio={1:1},lmargin=80.0pt,tmargin=90.0pt]{geometry}

\def\acts{\curvearrowright}
\usepackage{tikz}
\tikzset{anchorbase/.style={baseline={([yshift=-0.5ex]current bounding box.center)}}}
\usetikzlibrary{decorations.markings}
\usetikzlibrary{decorations.pathreplacing}
\usetikzlibrary{arrows,shapes,positioning,backgrounds}
\tikzstyle directed=[postaction={decorate,decoration={markings,
    mark=at position #1 with {\arrow{>}}}}]
\tikzstyle rdirected=[postaction={decorate,decoration={markings,
    mark=at position #1 with {\arrow{<}}}}]

\usepackage{graphicx}
\usepackage[vcentermath]{youngtab}
\usepackage{nicefrac}

 \newlength{\baseunit}               
 \newcount{\numlines}                
\newtheorem{theorem}[subsubsection]{Theorem}
\newtheorem{lemma}[theorem]{Lemma}
\newtheorem{prop}[theorem]{Proposition}
\newtheorem{corollary}[subsubsection]{Corollary}

\theoremstyle{definition}

\newtheorem{remark}[theorem]{Remark}
\newtheorem{question}[theorem]{Question}

\newtheorem{example}[subsubsection]{Example}

\newcommand{\im}{\mathrm{im}}
\newcommand{\tto}{\twoheadrightarrow}
\newcommand{\cO}{\mathcal{O}}
\newcommand{\cF}{\mathcal{F}}
\newcommand{\cC}{\mathcal{C}}
\newcommand{\cP}{\mathcal{P}}

\newcommand{\mN}{\mathbb{N}}
\newcommand{\mZ}{\mathbb{Z}}
\newcommand{\mR}{\mathbb{R}}
\newcommand{\mC}{\mathbb{C}}

\newcommand{\mL}{\mathbb{L}}
\newcommand{\fg}{\mathfrak{g}}
\newcommand{\ftg}{\tilde{\mathfrak{g}}}
\newcommand{\fk}{\mathfrak{k}}
\newcommand{\fa}{\mathfrak{a}}
\newcommand{\fs}{\mathfrak{s}}
\newcommand{\fn}{\mathfrak{n}}
\newcommand{\fb}{\mathfrak{b}}
\newcommand{\fh}{\mathfrak{h}}

\newcommand{\End}{\mathrm{End}}
\newcommand{\Hom}{\mathrm{Hom}}
\newcommand{\Res}{\mathrm{Res}}
\newcommand{\coker}{\mathrm{coker}}
\newcommand{\Ext}{\mathrm{Ext}}

\newcommand{\Ob}{\mathrm{Ob}}

\newcommand{\ba}{\mathbf{a}}

\newcommand{\bc}{\mathbf{c}}
\newcommand{\la}{\lambda}
\newcommand{\mf}{\mathfrak}

\newcommand{\vecc}{\mathsf{vec}}
\newcommand{\oa}{\bar{0}}
\newcommand{\ob}{\bar{1}}
\newcommand{\gmod}{\mbox{-gmod}}

\newcommand{\tL}{\mathsf{L}}

\newcommand{\hs}{{\fh^\fs}}

\begin{document}
\title[Semisimple super Takiff]{Representation theory of a semisimple extension of the Takiff superalgebra}
\author{Shun-Jen Cheng}
\address{Institute of Mathematics, Academia Sinica, Taipei, Taiwan 10617}
\email{chengsj@math.sinica.edu.tw}
\author{Kevin Coulembier}
\address{School of Mathematics and Statistics, University of Sydney, F07, NSW 2006, Australia}
\email{kevin.coulembier@sydney.edu.au}


\date{\today}

\begin{abstract}
We study a semisimple extension of a Takiff superalgebra which turns out to have a remarkably rich representation theory. We determine the blocks in both the finite-dimensional and BGG module categories and also classify the Borel subalgebras. We further compute all extension groups between two finite-dimensional simple objects and prove that all non-principal blocks in the finite-dimensional module category are Koszul.
\end{abstract}

\maketitle


\section*{Introduction}

 Let $\fs$ be a finite-dimensional simple Lie algebra and let $\Pi V$ be its adjoint module regarded as an odd space. We can build in a rather simple fashion the corresponding Takiff superalgebra $\fs\rtimes \Pi V$ by declaring $\Pi V$ to be an abelian ideal. Letting $\mathfrak d$ denote the space of outer derivations of $\fs\rtimes \Pi V$ we can show that $\fg=\left(\fs\rtimes \Pi V\right)\rtimes \mathfrak d$ is a semisimple Lie superalgebra, which is the semisimple extension of the Takiff superalgebra from the title. In this paper we attempt to understand the representation theory of $\fg$. Our motivation, among others, comes from the structure theory of classical Lie superalgebras.

A finite-dimensional Lie superalgebra is called `classical' (or `quasi-reductive' in \cite{Serga}) if the adjoint action restricted to its even subalgebra is semisimple. In particular, the even subalgebra is a reductive Lie algebra, which is one of the main reasons classical Lie superalgebras have been studied extensively in the literature. While many simple Lie superalgebras are classical, not all of them are. Indeed, from the classification of simple Lie superalgebras in \cite{Kac77} it follows that all simple Lie superalgebras are classical except for those of Cartan types, i.e., the Lie superalgebras of vector fields on purely odd dimensional superspaces.

It is known, see \cite[Theorem B]{El} (or \cite[Theorem 2.6]{CCC}), that all classical Lie superalgebras can be obtained from the semisimple classical Lie superalgebra ones by applying two types of elementary extensions, namely even central extensions and extensions by an odd abelian vector space similar to the one described for the Takiff superalgebra above.

Thus, one is led naturally to the study of semisimple classical Lie superalgebras. In contrast to Lie algebras, a semisimple Lie superalgebra is in general not a direct sum of simple Lie superalgebras, see \cite[Theorem 6]{Kac77} and \cite[Proposition 7.2]{Ch95}, and the same is true for classical semisimple Lie superalgebras. However, the building blocks of classical semisimple Lie superalgebras are precisely classical simple Lie superalgebras and the semisimple Lie superalgebra of the form $\fg$ above.

While the representation theory of simple Lie superalgebras has been studied in great detail by mathematicians and physicists alike and substantial progress has been made by now towards a satisfactory and complete theory, the representation theory of the semisimple extension of the Takiff superalgebra $\fg$ has not been studied much. While it is easy to compute the finite-dimensional irreducible characters of $\fg$, see \cite[Section 10]{Ch95} for even more general types of semisimple Lie superalgebras, not much about the finite-dimensional module category $\cF$ seems to be known. In this paper, we first classify and determine the blocks in $\cF$ and also in the corresponding BGG category $\cO$. When rank$\,\fs\ge 2$ we show that the number of blocks in either category is finite and equals the determinant of the Cartan matrix, which coincides with the number of minuscule representations of $\fs$ and also with the order of the fundamental group of $\fs$ (the quotient of the integral weight lattice with the root lattice). In the case of $\fs=\mathfrak{sl}_2$ one has three blocks in either module category. As a consequence of our block decomposition of category $\cO$ we prove that the centre of the universal algebra of $\fg$ is trivial.

Another main result of this work is the computation of all extension groups between two simple objects in $\cF$. Remarkably, our computation of the first extension groups in $\cF$ reveals a perfect match with the first extensions groups in the category of so-called finite current conformal modules first calculated in \cite[Section 4]{CKW}. The notion of conformal modules was originally motivated by $2$-dimensional conformal field theory. Our results on the extension groups show that the non-principal blocks in $\cF$ are indeed Koszul. Interestingly, the category $\cF$ is already graded Koszul itself and is in fact the graded lift in a `super' sense of the category of modules of its derived algebra $\fg'$. We also investigate the tensor powers $\left(\mC^{n|n}\right)^{\otimes r}$ in the case $\fs\cong\mathfrak{sl}(n)$ and study the $\fg$-invariants in the corresponding endomorphism algebras in the spirit of the celebrated Schur-Weyl duality. We show that this algebra of $\fg$-invariants is isomorphic to the group algebra of the symmetric group in $r$ letters in the case $r<n$ in Theorem \ref{ThmIT}, but that for high $r$ the symmetric group does not produce all invariants.

We conclude this introduction with an outline of the paper. In Section \ref{SecPrel}, we set up notations and also recall some basic results about Lie superalgebras that will be used throughout the paper. In Section \ref{SecTakiff}, the main object of study, the semisimple extension of the Takiff superalgebra $\fg$, is constructed. We classify the Borel subalgebras of $\fg$ up to conjugation by inner automorphisms. In particular, we show in Theorem \ref{ThmBorel} that the number of such conjugacy classes equals twice the number of the so-called $\oplus$-sign types indexed by the positive roots of the corresponding simple Lie algebra $\fs$ that have been studied earlier in \cite{S} in the context of left cells in affine Weyl groups. They admit remarkable combinatorial interpretations, e.g., in type $A$ these numbers are precisely the Catalan numbers. In Section \ref{Secblocks}, we determine the blocks decompositions of $\cF$ and $\cO$ and study invariant theory. Finally, in Section \ref{SecExt}, the higher extension groups between simple objects in $\cF$ are computed, from which Koszulity of the non-principal blocks is derived.


\section{Preliminaries}\label{SecPrel}
We always work over the ground field $\mC$ of complex numbers and set $\mZ$ and $\mN$ to be the sets of all and nonnegative integers, respectively.

\subsection{Classical Lie superalgebras}\label{DefClass}
We refer to \cite{Kac77, CW, Mu} for more general background on Lie superalgebras.
Recall that a finite-dimensional Lie superalgebra $\fg=\fg_{\oa}\oplus\fg_{\ob}$  is `classical' if the adjoint action of $\fg_{\oa}$ on $\fg$ is semisimple (in particular $\fg_{\oa}$ is reductive).

\subsubsection{The module category}\label{DefF}
For a classical Lie superalgebra $\fg$, we denote by $\cF(\fg)$ the category of finite-dimensional (super) modules which are semisimple over $\fg_{\oa}$. We denote the functor which switches the $\mZ/2$-parity of a module by $\Pi$.

Since the functors $U(\fg)\otimes_{U(\fg_{\oa})}-$ (respectively $\Hom_{U(\fg_{\oa})}(U(\fg),-)$) send any semisimple finite-dimensional $\fg_{\oa}$-module to a projective (respectively injective) module in $\cF(\fg)$, it follows that $\cF(\fg)$ has enough injective and projective objects. Since $\cF(\fg)$ is a tensor category, or because $U(\fg)\otimes_{U(\fg_{\oa})}-$ and $\Hom_{U(\fg_{\oa})}(U(\fg),-)$ are isomorphic up to composition with an auto-equivalence, see \cite{BF, Go}, it follows that injective and projective objects coincide.

Unless stated otherwise all morphisms are homogeneous with respect to the $\mZ/2$-grading.

\subsubsection{Borel subalgebras}\label{DefBorel} We follow the definition in \cite{PS} of Borel subalgebras.
Consider a classical Lie superalgebra $\fg$ and fix a Cartan subalgebra $\fh$ of $\fg_{\oa}$. We denote the set of roots (non-zero weights which appear in the adjoint action of $\fh$ on $\fg$) by $\Phi$.
An element $H\in\fh$ is called {\em regular}, if $\mathsf{{Re}}\beta(H)\not=0$, for all $\beta\in\Phi$. Here $\mathsf{Re}z$ stands for the real part of a complex number $z$. A regular element defines a triangular decomposition
\begin{align*}
\fg=\fg_{\mathsf{Re}<0}\oplus\fg_{\mathsf{Re}=0}\oplus\fg_{\mathsf{Re}>0},
\end{align*}
where $\fg_{\mathsf{Re}<0}$ denotes the span of vectors $X\in\fg$ such that $[H,X]=\beta(H)X$ with $\mathsf{Re}\beta(X)<0$, and similarly for $\fg_{\mathsf{Re}=0}$ and $\fg_{\mathsf{Re}>0}$. 
We call the subalgebra $\fg_{\mathsf{Re}=0}\oplus \fg_{\mathsf{Re}>0}$ arising from such a regular element a {\em Borel subalgebra} of $\fg$. Replacing a regular element $H$ by $-H$, we see that $\fg_{\mathsf{Re}=0}\oplus \fg_{\mathsf{Re}<0}$ is also a Borel subalgebra.
Observe that the even subalgebra of a Borel subalgebra $\fb$ is a Borel subalgebra of the reductive Lie algebra $\fg_{\oa}$.

\subsection{Relative Lie superalgebra cohomology}
\subsubsection{}For a Lie superalgebra $\fk$ with subalgebra $\fa$, we have the standard resolution of the trivial $\fk$-module
\begin{equation}\label{KosRes}\cdots\to U(\fk)\otimes_{U(\fa)}\wedge^i(\fk/\fa)\to\cdots\to U(\fk)\otimes_{U(\fa)}(\fk/\fa)\to U(\fk)\otimes_{U(\fa)}\mC\to\mC\to 0,\end{equation}
see \cite{Fuks}.
We will only need the differential in a specific case, so we do not write it out in full generality. By definition, the relative Lie superalgebra cohomology of a $\fk$-module $M$ is given by
$$H^k(\fk,\fa;M)\;:=\; H^k(\Hom_{\fk}(U(\fk)\otimes_{U(\fa)}{\wedge}^\bullet\fk/\fa,M))\,\cong\, H^k(\Hom_{\fa}(\wedge^\bullet\fk/\fa,M)).$$
In particular, for $\fk$-modules $M$ and $N$ we have
\begin{equation}\label{eqRH}H^k(\fk,\fa;\Hom_{\mC}(M,N))\;\cong\; H^k(\Hom_{\fa}(\wedge^\bullet\fk/\fa\otimes M,N)).\end{equation}

\begin{lemma}\label{LemRHom}
For a classical Lie superalgebra $\fg$, $M,N\in\cF(\fg)$ and $i\in\mN$ we have
$$\Ext^i_{\cF(\fg)}(M,N)\;\cong\; H^i(\fg,\fg_{\oa};\Hom_{\mC}(M,N))\;\cong\;H^i(\Hom_{\fg_{\oa}}(S^\bullet\fg_{\ob}\otimes M,N)).$$
If the action of $\fg_{\ob}$ on $N$ is trivial, the coboundary homomorphisms are
$$d:\Hom_{\fg_{\oa}}(S^{n-1}\fg_{\ob}\otimes M,N)\to \Hom_{\fg_{\oa}}(S^{n}\fg_{\ob}\otimes M,N):$$
$$ (d\alpha)(X_1X_2\cdots X_n\otimes v)= \sum_{i=1}^n\alpha( X_1X_2\stackrel{\hat{i}}{\cdots} X_n\otimes X_i v).$$
\end{lemma}
\begin{proof}
It suffices to observe that \eqref{KosRes} is a projective resolution of $\mC$ in $\cF(\fg)$ and that we can identify the space $\wedge^n(\fg/\fg_{\oa})$ with $S^n\fg_{\ob}$.
\end{proof}


\subsection{Classical Lie superalgebras of type I}
In this section, we completely ignore $\mZ/2$-parity, so every statement and module only makes sense modulo $\Pi$.

\subsubsection{} A classical Lie superalgebra $\fg$ is of type I if it admits a three-term $\mZ$-grading
$$\fg=\fg_{-1}\oplus \fg_0\oplus \fg_1,\quad\mbox{with}\quad\fg_{\oa}=\fg_{0}\mbox{ and }\fg_{\ob}=\fg_{-1}\oplus\fg_1.$$
We will set $\fg_{\le0}=\fg_0\oplus\fg_{-1}$ and $\fg_{\ge0}=\fg_0\oplus\fg_{1}$.
It is then well-known that $\cF(\fg)$ is an essentially finite highest weight category, see \cite{BS}. In particular, this is the case because it is a parabolic category $\cO$ for parabolic subalgebra $\fg_{\le0}$ or $\fg_{\ge 0}$, see, e.g., \cite[\S 3]{CCC}.

We choose a Cartan subalgebra $\fh$ and a Borel subalgebra $\fb_{\oa}$ of $\fg_{\oa}$ with $\fh\subset \fb_{\oa}$. For $\lambda\in\fh^\ast$ we have the simple highest weight module $L^0(\lambda)$ of $\fg_0$.
If  $L^0(\lambda)$ is finite dimensional, we have the (co)standard modules (with respect to parabolic subalgebra $\fg_{\ge0}$)
\begin{align*}
&\nabla(\la):={\rm coind}_{\fg_{\le0}}^\fg L^0(\la),\\
&\Delta(\la):={\rm ind}_{\fg_{\ge0}}^\fg L^0(\la),
\end{align*}
in $\cF(\fg)$. Denote by $L(\lambda)$ the top of $\Delta(\la)$. This is the simple highest weight module, with respect to the Borel subalgebra $\fb=\fb_{\oa}\oplus\fg_1$, with highest weight $\lambda$. It follows that the simple modules in $\cF(\fg)$ are labelled by the same set of weights as those in $\cF(\fg_{\oa})$.

Let $P(\la)$ denote the projective cover of $L(\la)$ in $\cF(\fg)$.  Then it has a filtration of submodules such that each quotient is isomorphic to a standard module and we have the BGG reciprocity:
\begin{equation}\label{BGGrec}
(P(\nu):\Delta(\kappa))\;=\;[\nabla(\kappa):L(\nu)].
\end{equation}

\subsubsection{}
We will also briefly consider the BGG category $\cO$ with respect to the Borel subalgebra $\fb=\fb_{\oa}\oplus\fg_1$ of $\fg$ above. For $\nu\in\fh^*$, let $\Delta^0(\nu)$ and $\nabla^0(\nu)$ denote the $\fg_0$-Verma and $\fg_0$-dual Verma modules, respectively. We define the $\fg$-Verma and $\fg$-dual Verma modules to be respectively:
\begin{align*}
&\nabla'(\nu):={\rm coind}_{\fg_{\le 0}}^\fg \nabla^0(\nu),\\
&\Delta'(\nu):={\rm ind}_{\fg_{\ge 0}}^\fg \Delta^0(\nu).
\end{align*}
Denote the projective cover of $L(\nu)$ in $\cO$ by $P'(\nu)$, which again has a $\Delta'$-filtration. We have for $\nu,\kappa\in\fh^*$ the BGG reciprocity in $\cO$:
\begin{align}\label{BGG:O}
(P'(\nu):\Delta'(\kappa))\;=\;[\nabla'(\kappa):L(\nu)].
\end{align}


\section{A semisimple extension of the Takiff superalgebra}\label{SecTakiff}

We fix a finite-dimensional simple Lie algebra $\fs$ with a fixed Cartan subalgebra $\hs$.

\subsection{Definitions}
We introduce the classical Lie superalgebra $\fg$ which we will study in the remainder of the paper.

\subsubsection{} Let $\wedge(\xi)$ be the associative Grassmann superalgebra in the odd indeterminate $\xi$ and let $\mf d$ denote the Lie superalgebra of derivations of $\wedge(\xi)$. Then $\mf d=\mC\partial_\xi+\mC\xi\partial_\xi$, where $\partial_\xi$ is the derivation of $\wedge(\xi)$ uniquely determined by $\partial_\xi(\xi)=1$.

We can form the Takiff superalgebra $\mf s\otimes \wedge(\xi)$.  Note that $\mf d$ acts naturally on the Lie superalgebra $\mf s\otimes\wedge(\xi)$, so that we may form the extension
\begin{equation}\label{Defg}
\mf g:=\left(\mf s\otimes\wedge(\xi)\right)\rtimes \mf d.
\end{equation}
The radical (maximal solvable ideal) of $\fg$ is trivial, so by definition $\mf g$ is a semisimple Lie superalgebra.  This is the Lie superalgebra $\tilde{\mathfrak{s}}^d$ in the notation of \cite[Example~5.1]{Serga}.

 In the special case $\mathfrak{s}=\mathfrak{sl}(2)$, we obtain $\fg\cong\mathfrak{pe}(2)$, the periplectic Lie superalgebra. Our results will therefore also recover a special case of the representation theory of that superalgebra, as studied for instance in~\cite{B+9, CC}.


\subsubsection{}
The adjoint action of $-\xi\partial_\xi$ provides a $\mZ$-grading $\mf g=\mf g_{-1}\oplus\mf g_0\oplus\mf g_{+1}$ with
\begin{align*}
\fg_{-1}=\mf s\otimes\xi,\quad\fg_0=\mf s\oplus\mC\xi\partial_\xi,\quad \fg_{+1}=\mC\partial_\xi.
\end{align*}
The Lie superalgebra $\fg$ is clearly classical of type I.

We will always consider the Cartan subalgebra $\fh:=\hs\oplus\mC\xi\partial_\xi$ of $\fg$.
We define the element $\delta\in\fh^\ast$ determined by $\delta(H)=0$ for $H\in\hs$ and $\delta(\xi\partial_\xi)=-1$.

Denote the roots of $\fs$ with respect to $\hs$ by $\Phi(\fs)$. Then, the roots of $\fg$ with respect to $\fh$ are as follows:
\begin{equation}\label{EqRoots}
\Phi=\Phi_{\oa}\sqcup\Phi_{\ob},\qquad\Phi_{\oa}=\{\alpha|\alpha\in\Phi(\fs)\},\quad \Phi_{\ob}=\{\alpha-\delta, \pm\delta|\alpha\in\Phi(\fs)\}.
\end{equation}

In the next subsection we will classify the Borel subalgebras, as defined in Section \ref{DefBorel}, of $\fg$. First we review the two canonical `extremal' cases.

\begin{example}\label{ExBorel}
Let $h\in\hs$ be a regular element giving rise to a Borel subalgebra $\fb^{\fs}$ of $\fs$. Choose $\la\in\mR$ with $|\mathsf{Re}\alpha(h)|<\la$, for all $\alpha\in\Phi(\fs)$. Then the element $H_+:=h+\la\xi\partial_\xi$ is regular, and the Borel subalgebra of $\fg$ associated with $H$ is $\fb^{\fs}+\mC\xi\partial_\xi+\fs\otimes\xi$, which gives a Borel subalgebra with maximal possible dimension.
On the other hand, the element $H_-:=h-\la\xi\partial_\xi$ is also regular and the associated Borel subalgebra of $\fg$ is $\fb^{\fs}+\mC\xi\partial_\xi+\mC\partial_\xi$, which gives a Borel subalgebra with minimal possible dimension.
\end{example}

\subsection{Borel subalgebras}
Our goal is to describe the Borel subalgebras of $\fg$ up to conjugacy  by automorphisms induced by inner automorphisms of $\fg_{\oa}$.

\subsubsection{}Let $\fb$ and $\fb'$ be two Borel subalgebras of $\fg$. Since their even subalgebras are Borel subalgebras of the even part, $\fb_{\bar 0}$ and $\fb'_{\bar 0}$ are conjugate by an { inner} automorphism of the reductive Lie algebra $\fg_0$. Such an automorphism induces an automorphism on $\fg$. Note that an inner automorphism of $\fg_{\oa}$ that preserves a Borel subalgebra also preserves the root spaces. Hence, in order to determine the Borel subalgebras of $\fg$ up to conjugacy by inner automorphisms, it suffices to determine the Borel subalgebras of $\fg$ such that $\fb_{\bar 0}=\fb^\fs\oplus\mC\xi\partial_\xi$ is a fixed Borel subalgebra of $\fg_0$. Here, $\fb^\fs$ is a fixed Borel subalgebra of $\fs$.

We thus fix a Borel subalgebra $\fh^{\fs}\subset\fb^{\fs}\subset \fs$ for the rest of the section. We also let $\mf{n}^\fs$ be the corresponding nilradical of $\fb^\fs$ and also let $\mf{n}^\fs_-$ denote the opposite nilradical so that $\fs=\fb^\fs\oplus\mf{n}^\fs_-$.

With this in place we can state our main result.

\begin{theorem}\label{ThmBorel}
\begin{enumerate} [label=(\roman*)]
\item
 Let $W$ be the Weyl group, $\mathsf{h}$ the Coxeter number, and $e_1,\ldots,e_r$ be the exponents of $\fs$, where $r=\mathrm{rank}\,\fs$. Then the number of conjugacy classes of Borel subalgebras of $\fg$ equals $$\frac{2}{|W|}\prod_{i=1}^r(\mathsf{h}+e_i+1).$$
Explicitly, these numbers are as follows:
\begin{center}
\begin{tabular}{  |p{2cm}|p{6cm}|} 
 \hline
 $\fs$ & conjugacy classes\\
 \hline\hline
$\mf{sl}(n)$  \vspace{1mm}& $\frac{2}{n+1}{{2n}\choose{n}}$ \vspace{1mm}\\
\hline
$\mf{so}(2n+1)$  \vspace{1mm}&$2{{2n}\choose{n}}$ \vspace{1mm} \\\hline
 $\mf{sp}(2n)$  \vspace{1mm}& $2{{2n}\choose{n}}$ \vspace{1mm} \\\hline
 $\mf{so}(2n)$ \vspace{1mm}&  $2\left({{2n}\choose{n}}-{{2n-2}\choose{n-1}}\right)$ \vspace{1mm}\\\hline
 $E_6$& $2^7\cdot 13$ \\\hline
$E_7$& $2^7\cdot 5\cdot 13$ \\\hline
$E_8$& $2^7\cdot 5\cdot 11\cdot 19$ \\\hline
$F_4$& $2\cdot 3\cdot 5\cdot 7$ \\\hline
 $G_2$& $2^4$ \\\hline
\end{tabular}

\end{center}
\item We have the following possibilities for $\fb_{\ob}$, the odd part of a Borel subalgebras $\fb$ with even subalgebra $\fb_{\bar 0}=\fb^\fs\oplus\mC\xi\partial_\xi$:
\begin{itemize}
\item[(a)]
$\fb_{\bar 1}=N\otimes\xi$, where $N$ is a $\fb^{\fs}$-ideal in $\fs$ such that $\fb^\fs\subseteq N$.
\item[(b)]  $\fb_{\bar 1}=N'\otimes\xi+\mC\partial_\xi$, where $N'$ is a $\fb^{\fs}$-ideal such that $N'\subseteq\mf{n}^\fs$.
\end{itemize}
\end{enumerate}
\end{theorem}

We start the proof with the following lemma.
\begin{lemma}\label{lem:borel:odd}
Every Borel subalgebras $\fb$ with even subalgebra $\fb_{\bar 0}=\fb^\fs\oplus\mC\xi\partial_\xi$ is of one of the two types described in Theorem~\ref{ThmBorel}(ii).

 \end{lemma}

\begin{proof}
Suppose that $H=h+\la\xi\partial_\xi$ is a regular element giving rise to the Borel subalgebra $\fb$. Then for any $\alpha\in\Phi(\fs)$ we have $\mathsf{Re}\alpha(h)=\mathsf{Re}\alpha(H)\not=0$, hence $h\in\hs$ is a regular element of $\hs\subseteq\fs$. Also, $\mathsf{Re}\delta(H)=-\mathsf{Re}\la\not=0$.

We have $\fb_{\bar 1}\subseteq\fg_{\bar 1}=\left(\fs\otimes\xi\right)\oplus\mC\partial_\xi$. It is clear that $\mC\partial_\xi\subseteq\fb_{\bar 1}$ if and only if $\mathsf{Re}\la<0$.

Now, suppose that $\mathsf{Re}\la>0$. In this case, we have $\fb_{\bar 1}=N\otimes\xi$, for some $\fb^\fs$-ideal $N\subseteq\fs$. Since $[H,x\otimes\xi]=\la (x\otimes\xi)$, for $x\in\hs$, it follows $\hs\subseteq N$ in this case. Since $N$ is $\fb^\fs$-invariant, it follows that $\fb^\fs\subseteq N$.

Now, if $\mathsf{Re}\la<0$, then we have $\fb_{\bar 1}=N'\otimes\xi+\mC\partial_\xi$, for some $\fb^\fs$-ideal $N'\subseteq\fs$. Since $[H,x\otimes\xi]=\la (x\otimes\xi)$, for $x\in\hs$, it follows $\hs\cap N'=0$ and thus, $N'\subseteq\mf{n}^\fs$.
\end{proof}

\begin{lemma}\label{lem:bijection}
There is a bijection between $\fb^\fs$-ideals in $\fs$ containing $\fb^\fs$ and $\fb^\fs$-ideals contained in $\mf{n}^\fs$.
\end{lemma}

\begin{proof}
We have a canonical bijection between $\fb^\fs$-ideals containing $\fb^\fs$ and $\fb^{\fs}$-submodules in~$\fs/\fb^{\fs}$. Via the Killing form, we get an isomorphism $(\fs/\fb^{\fs})^\ast\cong \fn^{\fs}$ of $\fb^{\fs}$-modules. We thus get a canonical bijection between $\fb^{\fs}$-submodules in~$\fs/\fb^{\fs}$ and $\fb^{\fs}$-submodules in $\fn^{\fs}$, by sending a submodule $K\subset \fs/\fb^{\fs}$ to the kernel of $\fn^{\fs}\tto K^\ast$.
\end{proof}




The $\fb^\fs$-ideals inside the nilradical $\mf{n}^\fs$ were studied earlier by Jian-Yi Shi who has obtained the following characterization.

\begin{theorem}\cite[Theorem 1.4]{S}\label{thm:shi}
Suppose that $N'\subseteq\mf{n}^\fs$ is a $\fb^\fs$-ideal. Then there exists an element $h\in\hs$, such that $\mathsf{Re}\alpha(h)>0$, for any positive root $\alpha$ of $\fs$ and furthermore:
\begin{itemize}
\item[(i)] If $\fs_{\alpha}\subseteq N'$, then $\mathsf{Re}\alpha(h)>1$.
\item[(ii)] If $\fs_{\alpha}\not\subseteq N'$, then $0<\mathsf{Re}\alpha(h)<1$.
\end{itemize}
\end{theorem}

\begin{prop}\label{prop:ideal=borel}
\begin{itemize}
\item[(i)]Let $N$ be a $\fb^\fs$-ideal of $\fs$ such that $\fb^\fs\subseteq N$. Then $\fb:=\fb^\fs+\mC\xi\partial_\xi + N\otimes\xi$ is a Borel subalgebra of $\fg$.
\item[(ii)]Let $N'$ be a $\fb^\fs$-ideal of $\fs$ such that $N'\subseteq\mf{n}^\fs$. Then $\fb:=\fb^\fs+\mC\xi\partial_\xi + N'\otimes\xi+\mC\partial_\xi$ is a Borel subalgebra of $\fg$.
\end{itemize}
\end{prop}

\begin{proof}
Let $N'\subseteq\mf{n}^\fs$ be a $\fb^\fs$-ideal. Let $\alpha$ be a positive root. For simplicity, let us say that $\alpha\in N'$ if $\fs_\alpha\subseteq N'$.
Then, by Theorem \ref{thm:shi} we can find an element $h\in\hs$ with $0<\mathsf{Re}\alpha(h)\not=1$, for every positive root $\alpha$, and $1<\mathsf{Re}\alpha(h)$ if and only if $\alpha\in N'$. Define
\begin{align*}
H':=h-\xi\partial_\xi.
\end{align*}
Then $H'$ is a regular element and for any $\alpha\in\Phi(\fs)$ we have $\mathsf{Re}\alpha(H')-\mathsf{Re}\delta(H')>0$ if and only if $\alpha\in N'$. As $[H',\partial_\xi]=\partial_\xi$, we see that the Borel subalgebra corresponding to this regular element is precisely $\fb:=\fb^\fs+\mC\xi\partial_\xi + N'\otimes\xi+\mC\partial_\xi$. This proves Part (ii).

Recall the correspondence between a $\fb^\fs$-ideal $N$ containing $\fb^\fs$ and an ideal $N'\subseteq\mf{n}^\fs$ in Lemma \ref{lem:bijection}. Let $h$ be the corresponding element in $\hs$ for $N'$. Consider the element
\begin{align*}
H:=h+\xi\partial_\xi.
\end{align*}
Then $H$ is regular and from the correspondence one checks that the Borel subalgebra associated with $H$ is precisely $\fb^\fs+\mC\xi\partial_\xi + N\otimes\xi$. This proves Part (i).
\end{proof}

\begin{proof}[Proof of Theorem~\ref{ThmBorel}]
Given a regular element and its associated Borel subalgebra $\fb$, the odd part $\fb_{\bar 1}$ must be of the form \ref{ThmBorel}(ii)(a) or  \ref{ThmBorel}(ii)(b) by Lemma \ref{lem:borel:odd}. By Proposition \ref{prop:ideal=borel} such odd subspaces are indeed all odd parts of Borel subalgebras with fixed $\fb_{\bar 0}$. Thus, the number of non-conjugate Borel subalgebras of $\fg$ with fixed even part equals precisely twice the number of $\fb^\fs$-ideals inside $\mf{n}^\fs$. The numbers of such ideals for different $\fs$ have been calculated in \cite[Theorems 3.2 and 3.6]{S}, and these numbers are precisely half of the ones listed in the theorem.

 Finally, the closed formulas for these ad-nilpotent ideals in terms of the order of the Weyl group, the Coxeter number, and the exponents have been obtained in \cite[Theorem 1]{CP} by relating these ideals to certain orbits in the corresponding affine Weyl groups.
\end{proof}

\begin{remark}
The number $\frac{1}{n+1}{{2n}\choose{n}}$ appearing in Theorem~\ref{ThmBorel} for $\fs=\mathfrak{sl}(n)$, is the Catalan number and it admits various combinatorial interpretations.
The most relevant interpretation for us is that $\frac{1}{n+1}{{2n}\choose{n}}$ is the number of monotonic lattice paths along the edges of a grid with $n \times n$ square cells, which do not pass above the diagonal. This allows to visualise the Borel subalgebras of $\fg$ elegantly. For instance one can picture the Borel subalgebras in the realisation of $\fg$ in \ref{Matrixg} below, by viewing the $n\times n$-matrix $C$ as such a grid. For $\fs$ of other classical types $BCD$, the numbers $2n\choose n$ and ${2n\choose n}-{2n-2\choose n-1}$ admit similar interpretations as numbers of certain lattice paths inside non-rectangular grids. This gives explicit realisations of such ad-nilpotent ideals of classical types, see \cite[Figures 5--7]{S}.
\end{remark}

\subsection{Some $\fs$-module morphisms}\label{Stuff}
The extension groups between simple $\fg$-modules will be described in Section~\ref{SecExt} in terms of the following $\fs$-module morphisms.
For a finite-dimensional simple $\fs$-module $V$ and $n\in\mZ_{>0}$ we set:
\begin{equation}\label{Stuffeq}D^n[V]: \;\;S^n(\fs)\otimes V\to S^{n-1}(\fs)\otimes V,\quad X_1X_2\cdots X_n\otimes v\mapsto \sum_{i=1}^n X_1X_2\stackrel{\hat{i}}{\cdots} X_n\otimes X_i v.\end{equation}

Recall $\hs$ denotes a Cartan subalgebra of $\fs$. For $\mu\in(\hs)^\ast$, let $V_\mu$ denote the $\mu$-weight space in $V$.

\begin{lemma}\label{LemSurj}
Fix $n\in\mZ_{>0}$ and set $D^n:=D^n[V]$.
\begin{enumerate}
\item The space $S^{n-1}(\hs)\otimes V_\mu\subseteq \im D^n $, for $\mu\not=0$.
\item If $V_0=0$, then $D^n $ is surjective.
\item Take $\mu\in(\hs)^\ast$ and $\alpha_i\in\Phi(\fs)\sqcup\{0\}$, for $1\le i\le n-1$, such that $\mu+\sum_{i\in S}\alpha_i\not=0$ for every subset $S$ of the interval $[1,n-1].$  Then $X_1X_2\cdots X_{n-1}\otimes v$ is in $\im D^n $ for $v\in V_\mu$ and $X_i\in\fs_{\alpha_i}$.
\item If $V\not=\mC$, then $\coker D^2 $ is either zero or else a trivial $\fs$-module.
\item The morphism $D^1$ is surjective if and only if $V\not=\mC$.
\end{enumerate}
\end{lemma}
\begin{proof} For $\mu\not=0$, we can choose a basis $\{H_{i}\}$ of $\hs$ such that $\mu(H_{1})=1$ and $\mu(H_{i})=0$ for $i\not=1$. Now take a basis of $S^{n-1}(\hs)\otimes V_\mu$ where each basis element has the form
$H_{1}^a f\otimes v$, for $0\le a\le n-1$ and $f$ a product of $n-1-a$ elements $H_{i}$, $i\not=1$. Since we have
$$D^n(H_{1}^{a+1}f\otimes v)=(a+1) H_{1}^af\otimes v,$$
part (1) follows.

Now we prove part (2). For each root in $\fs$, we choose a non-zero root vector. Together with our basis of $\hs$ for each weight vector $v$, this yields a basis $B_v$ of $\fs$ for each weight vector. We then have a basis of $S^{n-1}(\fs)\otimes V$ where each basis element is of the form $g\otimes v$, with $v\in V$ a weight vector and $g$ a monomial in the basis $B_v$. We refer to the number of factors in the monomial $g$ which are in $\hs$ as the `degree' of $g\otimes v$. We prove that each basis element is in the image of $D^n$, by (downwards) induction on the degree. If the degree is $n-1$, we have already proved the claim in the first paragraph. Assuming the claim is proved for degree $k$ and $g\otimes v$ is a basis element of degree $k-1$, then $D^n(H_1g\otimes v)$ is equal to $mg\otimes v$, for some $m>0$, plus a linear combination of basis elements of strictly higher degree. Since the latter are in the image of $D^n$ by induction hypothesis, so is $g\otimes v$. This concludes the proof.

Part (3) can be proved identically to part (2). By the assumption on the weights, no vector in $V_0$ will appear in the procedure.

For Part (4) we freely use that $D^2$ is an $\fs$-module morphism. We consider a highest weight vector $x=\sum_i X_i\otimes v_i$ in $\fs\otimes V $ of weight $\lambda\not=0$. If $\lambda\not\in\Phi(\fs)\sqcup\{0\}$, then $x$ is in $\im D^2$, by Part (3). Now assume that $\lambda\in\Phi(\fs)$. Then the above summation defining $x$ contains a term $X_\lambda\otimes u$, for some $X_\lambda\in\fs_\lambda$ (and hence $u\in V_0$), while all other terms are in $\im D^2$ by Part (3). Now, $u$ is a linear combination of vectors of the form $u':=X_{-\alpha_1}\cdots X_{-\alpha_k}v^+$, for $v^+$ a highest weight vector of $V$ and $X_{-\alpha_i}\in \fs_{-\alpha_i}$ for positive roots $\alpha_i$ and $k>0$.

We set $w=X_{-\alpha_2}\cdots X_{-\alpha_k}v^+$ and $Y=[X_{-\alpha_1}, X_{\lambda}]$. Then we have
$$X_{-\alpha_1}(X_\lambda\otimes w)\;=\; X_\lambda\otimes u'\,+\, Y\otimes w.$$
By Part (3), $X_\lambda\otimes w$ is in $\im D^2$, since $\lambda+\alpha_1\not=0$. The term $Y\otimes w$ (which might be zero) is also in $\im D^2$ by Part (3). Consequently $X_\lambda\otimes u'$, and hence also $X_\lambda\otimes u$ and $x$, lies in $\im D^2$. Thus, any singular vector of non-trivial weight lies in $\im D^2$, and Part (4) follows.

Part (5) is obvious.
\end{proof}

\begin{remark}
 For $\fs=\mathfrak{sl}(2)$, the morphisms $D^i[\fs]$ are not surjective. For instance $\coker D^2[\fs]$ contains a copy of the trivial representation and $\coker D^3[\fs]$ contains a copy of the adjoint representation.
\end{remark}


\section{Character formulae, block decomposition and invariant theory}\label{Secblocks}

In this section we start the study of the representation theory of the Lie superalgebra $\fg$, as defined in \eqref{Defg}, in the category $\cF(\fg)$ of Section \ref{DefF}.

\subsection{Some reductions}
We reduce the problem of studying $\cF(\fg)$ to an equivalent but `smaller' problem. We will typically denote a weight in $\fh^\ast$ by $\lambda+a\delta$, with $\lambda\in (\fh^{\fs})^\ast$ and $a\in\mC$ for the root $\delta\in\Phi(\fg)$.

Note that the category $\cF(\fg)$ does not depend on the choice of Borel subalgebra. When we label the simple modules by their highest weight or consider (co)standard modules we will always take the Borel subalgebra
$$\fb:=\fb^{\fs}\oplus\mC\xi\partial_\xi\oplus\mC\partial_\xi=\fb_{\oa}\oplus\fg_1$$
of minimal dimension, see Example~\ref{ExBorel}.
\subsubsection{Integrability}
Consider a reductive algebraic group $S$ over $\mC$ with underlying Lie algebra $\fs$. We will focus on finite-dimensional $\fg$-modules for which the $(\fs\oplus\mC)\cong\fg_0$-module structure integrates to $S\times \mathbb{G}_m$, with $\mathbb{G}_m=\mC^\times$ the multiplicative group. In other words, we consider the category of modules $\cF(\fg)_{int}$ for the supergroup pair $(S\times\mathbb{G}_m,\fg)$. Concretely, only weights $\lambda+a\delta$ with $\lambda$ integral and $a\in\mZ$ appear in the modules in $\cF(\fg)_{int}$. We denote this set of weights by $\Lambda\subset\fh^\ast$.

It follows easily that
$$\cF(\fg)\;\cong\;\bigoplus_{x\in\mC/\mZ} \cF(\fg)_{int},$$
where $\mC/\mZ$ is the quotient of $\mC$ with respect to the canonical action $\mZ\acts\mC$.

\subsubsection{Parity}
We define the function
$$p:\;\Lambda\to \mZ/2,\quad \lambda+a\delta\mapsto a\hspace{-2mm}\mod 2.$$
It follows immediately from~\eqref{EqRoots} that
$$\cF(\fg)_{int}\cong\cF\oplus\cF', \qquad\mbox{with }\; \Pi:\cF\stackrel{\sim}{\to}\cF'$$
where $\cF$ is the full subcategory of modules where vectors of weight $\kappa\in\Lambda$ have degree $p(\kappa)$. The category $\cF'$ can be defined similarly, or as the image of $\cF$ under $\Pi$. Without loss of generality we will therefore focus on $\cF$. Simple highest weight modules $L(\kappa)$ will be assumed to be in $\cF$, meaning we let the highest weight vector be of parity $p(\kappa)$.  We follow the corresponding convention for $\fg_0$-modules, which means that $L^0(\lambda+a\delta)$ is of the same parity as $a$.

In $\cF$ simple highest weight modules are uniquely defined up to isomorphism by their highest weights and we denote by $\Lambda^+\subset\Lambda$ the corresponding set of integral dominant weights.
 In a way analogous to the definition of $\cF$ above we define the corresponding BGG category $\cO$ of those modules which only have non-zero weight spaces for weights in $\Lambda$ with the same parity condition. This reduction is justified by the same arguments.

To make a distinction between $\fg_0$-modules and $\fs$-modules we will denote the simple $\fs$-module with highest weight $\lambda\in (\hs)^\ast$ by $L^0_\lambda$ and always consider it as purely even.

\subsection{Character formulae}
In this subsection we determine the combinatorics of the structural modules in the highest weight category $\cF$. Note that we derived above that $\cF(\fg)$ is a direct sum of categories equivalent to $\cF$. Simple modules in $\cF$ are uniquely determined by their highest weight $\lambda+a\delta\in\Lambda^+$.

\begin{lemma}\label{LemSimp}
If $\lambda\not=0$, we have $\nabla(\lambda+a\delta)=L(\lambda+a\delta)$. Moreover, there is a short exact sequence
$$0\to L(a\delta)\to \nabla(a\delta)\to L((a-1)\delta)\to 0.$$
Consequently, we have
$$\Res^{\fg}_{\fg_0}L(a\delta)\cong L^0(a\delta)\quad\mbox{and}\quad \Res^{\fg}_{\fg_0} L(\lambda+a\delta)\cong L^0(\lambda+a\delta)\oplus  L^0(\lambda+(a-1)\delta).$$
\end{lemma}
\begin{proof}
As a $\fg_0$-module, $\nabla(\lambda+a\delta)$ is clearly the sum of $L^0(\lambda+a\delta)$ and $L^0(\lambda+(a-1)\delta)$, where $\partial_\xi$ acts as an $\fs$-linear morphism from $L^0(\lambda+(a-1)\delta)$ to $L^0(\lambda+a\delta)$. For an arbitrary $v\in L^0(\lambda+(a-1)\delta)$ and $X\in\fs$, we have $(X\otimes\xi) v=0$ and hence
$$(X\otimes \xi)\partial_\xi v\;=\; [X\otimes \xi,\partial_\xi]v\;=\;-Xv.$$
Consequently, $\nabla(\lambda+a\delta)$ is simple if and only if $\lambda\not=0$.
\end{proof}


\begin{corollary}\label{CorBGG} Assume $\lambda\not=0$.
\begin{enumerate}
\item  We have $P(\lambda+a\delta)=\Delta(\lambda+a\delta)$, and furthermore a short exact sequence
$$0\to \Delta((a+1)\delta)\to P(a\delta)\to \Delta(a\delta)\to 0.$$
\item We have
$$
\label{EqDual}
L(\lambda+a\delta)^\ast\cong L(-w_0\lambda+(1-a)\delta)\quad\mbox{and}\quad L(a\delta)^\ast\cong L(-a\delta),
$$
where $w_0$ is the longest element of the Weyl group of $\fs$.
\end{enumerate}
\end{corollary}
\begin{proof}
These are immediate consequences of Lemma~\ref{LemSimp} and Equation~\eqref{BGGrec}.
\end{proof}

\begin{remark}\label{RemSimp}
Lemma~\ref{LemSimp} allows us to describe the simple $\fg$-module $L(\lambda+a\delta)$ with $\lambda\not=0$ as follows.
We have $\fs$-module morphisms $\sigma_i:L^0_\lambda\to L(\lambda+a\delta)$ for $i\in\{0,1\}$, with $\sigma_i$ of degree $i$, which yield an isomorphism
$$\Pi^{p(a)}L^0_\lambda\oplus \Pi^{p(a-1)}L^0_\la\,\stackrel{\sim}{\to}\, L(\lambda+a\delta),\qquad (\Pi^{p(a)}u,\Pi^{p(a-1)} v)\mapsto \sigma_{p(a)}(u)+\sigma_{p(a-1)}(v).$$
The remainder of the $\fg$-action is described by ($X\in\fs$):
\begin{align*}
&\xi\partial_{\xi}(\sigma_{p(a-i)}(u))=(a-i)\sigma_{p(a-i)}(u),\quad (X\otimes \xi)(\sigma_{p(a)}(u))=\sigma_{p(a-1)}(Xu),\\
&\partial_\xi(\sigma_{p(a-1)}(v))=\sigma_{p(a)}(v), \qquad (X\otimes \xi)(\sigma_{p(a-1)}(v))=0=\partial_\xi(\sigma_{p(a)}(u)).
\end{align*}
\end{remark}


\begin{lemma}\label{LemDelta}
For all $\lambda+a\delta,\mu+b\delta\in\Lambda^+$, we have
$$[\Delta(\lambda+a\delta):L(\mu+b\delta)]\;=\; \begin{cases}  \sum_{i=0}^{a-b}(-1)^{a-b-i}[\wedge^i\fs\otimes L^0_\lambda:L^0_\mu]&\mbox{ if $b\le a$ and $\mu\not=0$}\\
\left[\wedge^{a-b}\fs\otimes L^0_\lambda:\mC\right]&\mbox{ if $b\le a$ and $\mu=0$}\\
0&\mbox{if $b> a$}
\end{cases}$$
\end{lemma}
\begin{proof}
We have
$$\mathrm{res}^{\fg}_{\fg_0}\Delta(\lambda+a\delta)\;\cong\; \bigoplus_{i=0}^{\dim \fs}\wedge^i\fs\otimes L^0(\lambda+(a-i)\delta).$$
Since we already know $\mathrm{res}^{\fg}_{\fg_0} L(\mu+b\delta)$, the lemma follows easily.
\end{proof}

\begin{prop}\label{PropCartan}
If $\lambda\not=0$, we have
$$[P(\lambda+a\delta):L(\mu+b\delta)]\;=\; \begin{cases} \sum_{i=0}^{a-b}(-1)^{a-b-i}[\wedge^i\fs\otimes L^0_\lambda:L^0_\mu]&\mbox{ if $b\le a$ and $\mu\not=0$}\\
\left[\wedge^{a-b}\fs\otimes L^0_\lambda:\mC\right]&\mbox{ if $b\le a$ and $\mu=0$}\\
0&\mbox{if $b> a$}
\end{cases}$$
Furthermore, we have
$$[P(a\delta):L(\mu+b\delta)]\;=\; \begin{cases}  [\wedge^{a-b+1}\fs:L^0_\mu]&\mbox{ if $b\le a$ and $\mu\not=0$}\\
\left[\wedge^{a-b}\fs:\mC\right]+\left[\wedge^{a-b+1}\fs:\mC\right]&\mbox{ if $b\le a$ and $\mu=0$}\\
\delta_{\mu,0}&\mbox{if $b=a+1$}\\
0&\mbox{if $b> a+1$}
\end{cases}$$

\end{prop}

\begin{proof}
This is an immediate application of Corollary~\ref{CorBGG}(1) and Lemma~\ref{LemDelta}.
\end{proof}

\begin{remark}
A description of the $\fg_0\oplus\fg_{+1}$-action on the structural modules in $\cF$ is also given in \cite[Example~9.11]{Serga}.
\end{remark}
\subsection{Block decomposition}

For $\la+a\delta$ and $\mu+b\delta$ in $\Lambda^+$, we say that $\la+a\delta$ and $\mu+b\delta$ are {\em$\cF$-linked} if  $L(\la+a\delta)$ and $L(\mu+b\delta)$ lie in the same block in the category $\mathcal F$. We write $$\la+a\delta\sim\mu+b\delta.$$

Similarly, for $\la+a\delta$ and $\mu+b\delta$ in $\Lambda$, we say that $\la+a\delta$ and $\mu+b\delta$ are {\em $\cO$-linked} if  $L(\la+a\delta)$ and $L(\mu+b\delta)$ lie in the same block in the category $\mathcal O$. Naturally, two dominant weights are $\cO$-linked, if they are $\cF$-linked. We shall see in Corollary \ref{cor:blocks:O} and Proposition \ref{prop:block:sl2} below that the converse is also true. When there is no confusion we shall simply say `linked' when we mean $\cF$-linked. By Corollary \ref{CorBGG}(1) or the fact that it has simple top, every composition factor of $\Delta(\la+a\delta)$ is linked to $\la+a\delta$. We will use this fact freely. We also set
$$\rho_0=\frac{1}{2}\sum_{\alpha\in\Phi^+(\fs)}\alpha.$$

\begin{lemma}\label{lem:tensor}Assume that ${\rm rank}\,\mf s\ge 2$. Let $\la$ be a dominant $\fs$-weight such that for every root $\beta$ of $\mf s$ the weight $\la+\beta$ is dominant. Then
$$\lambda+a\delta\,\sim\,\lambda+b\delta\quad\mbox{and}\quad \lambda+a\delta\sim \lambda+\beta+b\delta,$$
for all $a,b\in\mZ$ and $\beta\in\Phi(\fs)$.
\end{lemma}

\begin{proof}
We claim that
\begin{align*}
L^0_\lambda\otimes\mf s\cong \bigoplus_{\beta\in\Phi(\mf s)}L^0_{\la+\beta} \oplus (L^0_\lambda)^{{\rm rank}\,\mf s}.
\end{align*}
 By Lemma~\ref{LemDelta}, we thus see that
$$[\Delta(\la+a\delta):L(\la+(a-1)\delta)]\,=\,[\fs\otimes L^0_\lambda:L^0_\lambda]-1\,=\, \mathrm{rank}\,\fs-1 \,\ge\, 1.$$
In particular $\lambda+a\delta\sim \lambda+(a-1)\delta$ and the first relation in the lemma follows by iteration. Similary, we find
$$[\Delta(\la+(b+1)\delta):L(\la+\beta+b\delta)]\,=\,[\fs\otimes L^0_\lambda:L^0_{\lambda+\beta}]\,=\, 1.$$
Hence $\lambda+\beta+b\delta$ is linked to $\lambda+(b+1)\delta$ and composition with the first relation yields the second relation in the lemma.

For completeness we prove the above claim. By the Weyl character formula we have
\begin{align*}
{\rm ch} \left(L^0_\lambda\otimes\mf s\right) = \frac{\sum_{w\in W}(-1)^{\ell(w)} w(e^{\la+\rho_0})}{\prod_{\alpha\in\Phi^+(\mf s)}(1-e^\alpha)}\times \left(\sum_{\beta\in\Phi(\mf s)}e^\beta + {\rm rank}\,\mf s\cdot 1\right).
\end{align*}
Since ${\rm ch}\mf s$ is $W$-invariant, we deduce that
\begin{align*}
{\rm ch} \left(L^0_\lambda\otimes\mf s\right) =& \sum_{\beta\in\Phi(\mf s)}\frac{\sum_{w\in W}(-1)^{\ell(w)} w( e^{\la+\beta+\rho_0})}{\prod_{\alpha\in\Phi^+(\mf s)}(1-e^\alpha)}  + {\rm rank}\,\mf s \cdot\frac{\sum_{w\in W}(-1)^{\ell(w)} w(e^{\la+\rho_0})}{\prod_{\alpha\in\Phi^+(\mf s)}(1-e^\alpha)}\\
&=\sum_{\beta\in\Phi(\mf s)}{\rm ch} L^0_{\la+\beta} + {\rm rank}\,\mf s\cdot{\rm ch}L^0_\lambda.
\end{align*}
This completes the proof.
\end{proof}

\begin{lemma}\label{lem:diffab}
Assume that ${\rm rank}\,\mf s\ge 2$ and $\la+a\delta\in\Lambda^+$. Then for any $b\in\mZ$ we have $$\la+a\delta\sim\la+b\delta\sim \la+2\rho_0+b\delta.$$
\end{lemma}

\begin{proof}
We set $\ell:=|\Phi^+(\fs)|$. We have
$$[\wedge^{\ell}\fs\otimes L^0_\lambda:L^0_{\la+2\rho_0}]=1\quad\mbox{and}\quad [\wedge^{i}\fs\otimes L^0_\lambda:L^0_{\la+2\rho_0}]=0,\quad\mbox{for $i<\ell$}.$$
Lemma~\ref{LemDelta} thus implies that
$$[\Delta(\la+a\delta):L(\la+2\rho_0+(a-\ell)\delta)]=1.$$
By iteration, this means that
$$\la+a\delta\sim \lambda+2k\rho_0+(a-k\ell)\delta,$$
for every $k\ge 0$.

For every $\la$, we can choose $k\ge 0$ such that $\lambda+2k\rho_0$ is a weight satisfying the conditions in Lemma~\ref{lem:tensor}.
Consequently, every relation in the lemma can be obtained by composing the above relation with the one in Lemma~\ref{lem:tensor}.
\end{proof}

Recall that a finite-dimensional irreducible $\mf s$-module is called {\em minuscule}, if every weight is $W$-conjugate to the highest weight. By $W$-invariance, every weight space is therefore one-dimensional. We will also say that an $\fh^{\fs}$-weight is minuscule if it is the highest weight of a minuscule representation.

\begin{lemma}\label{lem:minuscule}
Let $\la$ be a dominant integral $\fs$-weight. If $\la$ is not minuscule, then there exist $k_\alpha\in\mN$, $\alpha\in\Phi^+(\mf s)$, such that $\la-\sum_{\alpha>0}k_\alpha\alpha$ is dominant and not all $k_\alpha=0$.
\end{lemma}

\begin{proof}
Since $\la$ is not minuscule, there exists a weight $\mu$ in $L^0_\lambda$ such that $\mu$ is not in the $W$-orbit of $\la$. Now there exists $w\in W$ such that $w\mu$ is dominant and $w\mu\not=\la$. But $w\mu$ is also a weight of $L^0_\lambda$ and hence it must be of the desired form $\la-\sum_{\alpha>0}k_\alpha\alpha$, and not all $k_\alpha=0$.
\end{proof}

\begin{lemma}\label{LemMin2}
Assume that ${\rm rank}\,\mf s\ge 2$. Suppose that $\la$ and $\la-\sum_{\alpha>0}k_\alpha\alpha$ with $k_\alpha\in\mN$ are both dominant integral weights. Then for all $a,b\in\mZ$ we have
\begin{align*}
\la+a\delta\sim\la-\sum_{\alpha>0}k_\alpha\alpha +b\delta.
\end{align*}
\end{lemma}

\begin{proof}
We prove this by induction on $k:=\sum_\alpha k_\alpha$. The case $k=0$ is included in Lemma~\ref{lem:diffab}. Assume the lemma has been proved for $k-1$, assume $k_\beta>0$ for some $\beta\in \Phi^+$ and set $k_\alpha'= k_\alpha-\delta_{\alpha\beta}$.



We can find a $j\in\mN$ sufficiently large such that $\la+2j\rho_0-\sum_{\alpha}k'_\alpha\alpha+\gamma$ is dominant for all $\gamma\in\Phi(\fs)$. By Lemma~\ref{lem:tensor} we find
\begin{align}\label{eq:aux11}
\la+2j\rho_0-\sum_{\alpha}k'_\alpha\alpha +a\delta\sim \la+2j\rho_0-\sum_{\alpha}k_\alpha\alpha+b\delta.
\end{align}
By the induction hypothesis applied to $\la+2j\rho_0$ and $k'=\sum_{\alpha}k'_\alpha$, the weight on the left hand side of \eqref{eq:aux11} is linked to $\la+2j\rho_0+a\delta$. The latter weight is linked to $\la+a\delta$ by Lemma~\ref{lem:diffab}. The weight on the right hand side in \eqref{eq:aux11} is linked to $\la-\sum_{\alpha}k_\alpha\alpha+b\delta,$ again by Lemma~\ref{lem:diffab}. This concludes the proof.
\end{proof}


\begin{theorem}\label{thm:class:fd}
Assume that ${\rm rank}\,\mf s\ge 2$. The blocks in $\mathcal F$ are parameterised by the minuscule $\mf s$-weights. Indeed, let $\{\nu_i|1\le i\le r\}$ be set of minuscule weights of $\fs$. Then the highest weights of the simple $\fg$-modules in $\cF$ in the block $\cF_{\nu_i}$ containing {$L(\nu_i)$} are precisely the dominant weights of the form $\nu_i+\sum_{\alpha}k_\alpha\alpha+a\delta$, $k_\alpha\in\mN $ and $a\in\mZ$.
\end{theorem}

\begin{proof}
Clearly, if two weights $\la+a\delta$ and $\mu+b\delta$ in $\Lambda$ are linked, then $\la$ and $\mu$ differ by an element in the root lattice. So the number of blocks must be at least the index of the root lattice in the integral weight lattice.

On the other hand, if $\la+a\delta\in\Lambda^+$, then $\la+a\delta\sim\nu+a\delta$, for some minuscule highest weight $\nu$ by Lemmas \ref{lem:minuscule} and \ref{LemMin2}. Now, the number of minuscule representations of $\mf s$ is well known to be equal to the index of the root lattice in the integral weight lattice.
\end{proof}

\begin{corollary}\label{cor:blocks:O}
Assume that ${\rm rank}\,\mf s\ge 2$. The blocks in $\mathcal O$ are also parameterised by the minuscule weights of $\mf s$. Indeed, let $\{\nu_i|1\le i\le r\}$ be set of minuscule weights of $\fs$. Then the highest weights of the simple $\fg$-modules in $\cO$ in the block $\cO_{\nu_i}$ containing {$L(\nu_i)$} are precisely the weights in $\nu_i+\mZ\Phi(\fg)$.
\end{corollary}

\begin{proof}
 If two weights $\la+a\delta$ and $\mu+b\delta$ are linked, then $\la-\mu$ must lie in the root lattice. Thus, by Theorem \ref{thm:class:fd} there are at least as many blocks in $\cO$ as there are in $\cF$.

Now, let $\la+a\delta\in\Lambda$ and $\ell=|\Phi^+(\fs)|$.  We note that
$$[\Delta'(\la+a\delta):L(\la+2\rho_0+(a-\ell)\delta)]= 1,$$ since $\dim\Delta'(\la+a\delta)_{\la+2\rho_0+(a-\ell)\delta}=1$ and there are no weights in $\Delta'(\la+a\delta)$ of the form $\la+2\rho_0+(a-\ell+1)\delta$ or $\la+2\rho_0+\alpha+(a-\ell)\delta$, $\alpha\in\Phi^+(\fs)$.
Iterating, we conclude that $\la+a\delta$ is linked to $\la+2k\rho_0+(a-k\ell)\delta$, for any $k\ge 0$. Since we can choose $k\gg 0$ so that $\la+2k\rho_0$ is dominant, we conclude that every block in $\cO$ contains a simple module in $\cF$. The corollary now follows.
\end{proof}

The following proposition describes the blocks in the case ${\rm rank}\,\fs=1$, i.e., $\fs\cong \mathfrak{sl}(2)$. Let $\omega$ denote the fundamental weight and write $[n]=n\omega$, for $n\in\mZ$.

\begin{prop}\label{prop:block:sl2}
Let $\mf s\cong \mathfrak{sl}(2)$. Then the number of blocks in $\mathcal F$ and $\mathcal O$ equals $3$. Indeed, the highest weights of the simple $\fg$-modules in $\cF$ in these three blocks are:
\begin{align}\label{blocks:sl2}
\begin{split}
&\{[2n]+a\delta|n\in\mN ,a\in\mZ\},\\
&\{[1+2n]+(2a-n)\delta|n\in\mN ,a\in\mZ\},\\
&\{[1+2n]+(2a-n-1)\delta|n\in\mN ,a\in\mZ\}.
\end{split}
\end{align}
Furthermore, the highest weights of the simple $\fg$-modules in $\cO$ in the three blocks are given by the exact same formula as in \eqref{blocks:sl2} with $n\in\mN $ replaced in there by $n\in\mZ$.
\end{prop}

\begin{proof}
We first consider the case of $\cF$.
By Lemma~\ref{LemDelta}, we have the following identity in the Grothendieck group $K_0(\cF)$ for $n\ge 2$:
\begin{align*}
&\Delta([n]+a\delta)=L([n]+a\delta)+L([n+2]+(a-1)\delta)+L([n-2]+(a-1)\delta)+L([n]+(a-2)\delta),\\
&\Delta([1]+a\delta)=L([1]+a\delta)+L([3]+(a-1)\delta)+L([1]+(a-2)\delta),\\
&\Delta(a\delta)=L(a\delta)+L([2]+(a-1)\delta)+L((a-3)\delta).
\end{align*}
Now we can use Corollary \ref{CorBGG}(1) to prove the following:
\begin{align*}
&a\delta\sim[2n]+b\delta,\quad\forall n\in\mN ; a,b\in\mZ,\\
&[1]+a\delta\sim[1]+(a-2)\delta,\quad\forall a\in\mZ,\\
&[1]+a\delta\sim[1+2n]+(a-n)\delta,\quad\forall n\in\mN ;a\in\mZ.
\end{align*}
Furthermore, the weights $a\delta$, $[1]+a\delta$ and $[1]+(a-1)\delta$ are not linked. 
This proves the proposition for $\mathcal F$.

Now, we turn to the category $\mathcal O$ case. We want to prove that the linkage classes are given by the three sets in \eqref{blocks:sl2} now with $n\in\mN $ in there replaced by $n\in\mZ$.

Since the $\fg$-Verma module $\Delta'([2n]+a\delta)$ contains a composition factor $L([-2n-2]+a\delta)$, for $n\ge 0$, it follows from the linkage in $\cF$ that the set of weights $\{[2n]+a\delta|n,a\in\mZ\}$ forms one linkage class.

Next, consider the subcategory of $\mathcal O$ generated by simple objects with highest weights of the form $[1+2n]+a\delta$, for $n\in\mZ$ and $a\in\mZ$. Let $s\in W$ be the generator of the Weyl group of $\fs$. We compute
\begin{align*}
&s\cdot[1+2n]+(2m-n-\kappa)\delta=[1+2(-2-n)]+\left(2(m-n-1)+2+n-\kappa\right)\delta,\quad \kappa=0,1.
\end{align*}
This implies that each of the two sets of weights $\{[1+2n]+(2m-n)\delta|m,n\in\mZ\}$ and $\{[1+2n]+(2m-n-1)\delta|m,n\in\mZ\}$ are invariant under the dot action of the Weyl group.
Also, for these weights we have
\begin{align}\label{sl2:aux1}
\begin{split}&{\rm ch}\nabla'(\kappa)={\rm ch}L(\kappa)+{\rm ch}L(s\cdot\kappa), \quad\text{ if }\kappa\text{ is dominant},\\
&\nabla'(\kappa)=L(\kappa), \qquad\qquad\qquad\qquad\text{ otherwise}.
\end{split}
\end{align}
By \eqref{BGG:O}, we have therefore
\begin{align}\label{sl2:aux2}
\begin{split}
&P'(\kappa)=\Delta'(\kappa),\qquad\qquad\qquad\qquad\text{ if }\kappa\text{ is dominant},\\
&{\rm ch}P'(\kappa)={\rm ch}\Delta'(\kappa)+{\rm ch}\Delta'(s\cdot\kappa),\quad\text{otherwise.}
\end{split}
\end{align}
We compute the character of the Verma module:
\begin{align*}
{\rm ch}\Delta'(\kappa)={\rm ch}\nabla'(\kappa)+{\rm ch}\nabla'(\kappa+[2]-\delta)+{\rm ch}\nabla'(\kappa-[2]-\delta)+ {\rm ch}\nabla'(\kappa-2\delta).
\end{align*}
Thus, by \eqref{sl2:aux1} and \eqref{sl2:aux2}, we conclude that the highest weights of the composition factors of a Verma module with highest weight belonging to one of these two sets will all belong to the same set again, which implies that the same is true for a projective cover as well. This implies that this subcategory splits into two blocks with simple objects having highest weights of the forms described above.
\end{proof}

 Below we shall show that $U(\mathfrak g)$ has trivial center. We shall need the following.
\begin{prop}\label{prop:sergeev}\cite[Proposition~1.1]{Se2} Let $\mathfrak k$ be a finite-dimensional Lie superalgebra, $V$ a finite-dimensional $\mathfrak k$-module, $W\subseteq V$ a subspace, and $w_0\in W_{\oa}$ such that
the map $\mathfrak k\times W\rightarrow V$ given by
\begin{align*}
(X, w)\mapsto Xw_0+w, \quad X\in\mathfrak k,w\in W,
\end{align*}
is surjective. Then the restriction map $S(V^*)^{\mathfrak k}\rightarrow S(W^*)$ is injective.
\end{prop}

\begin{prop}
We have $\mathcal{Z}(\fg)\cong \mC\cong S(\fg)^{\fg}$, where $\mathcal{Z}(\fg)$ denotes the centre of $U(\fg)$ and $S(\fg)^{\fg}$ the algebra of $\fg$-invariants under the adjoint action on the symmetric algebra $S(\fg)$.
\end{prop}
\begin{proof}
Consider the vector space decomposition
$$\fg\;=\; \fg_{-1}\oplus \fn^{\fs}_-\oplus \fh\oplus \fn^{\fs}\oplus\fg_1.$$
We choose an ordered basis of $\fg$ which is compatible with the above decomposition.
This yields a PBW basis of $U(\fg)$ such that every basis element of $U(\fg)$ is a monomial of basis elements in $\fg$, each one in one of the above subspaces, and from left to right in the monomial we first have elements of $\fg_{-1}$, then of $\fn^{\fs}_-$ etc. In particular this yields a commutative diagram of $\mC$-linear morphisms
$$\xymatrix{
U(\fg)\ar[r] &U(\fg_{\le 0})\ar[r]& U(\fg_0)\ar[r]& U(\fh)\\
S(\fg)\ar[u]^{\sim}\ar[r] &S(\fg_{\le 0})\ar[r]\ar[u]^{\sim}& S(\fg_0)\ar[u]^{\sim}\ar[r]& S(\fh).\ar@{=}[u]
}$$
The vertical isomorphisms are the realisations of the PBW theorem, with respect to our chosen basis (which also yields bases of $\fg_{\le 0}$ and $\fg_0$). The upper horizontal arrows are Harish-Chandra morphisms. Concretely, the first one is a projection with kernel $U(\fg)\fg_{1}$, the second one has kernel $\fg_{-1}U(\fg_{\le 0})$ and the final one has kernel $\fn^{\fs}_-U(\fg_0)+U(\fg_0)\fn^{\fs}$. The lower horizontal arrows are the obvious surjective algebra morphisms. For instance, the first one is the algebra morphism induced by the $\mC$-linear projection $\fg\tto \fg_{\le 0}$ with kernel $\fg_1$.

Now we apply Proposition \ref{prop:sergeev} above with $\mathfrak k=\mathfrak g$, $V=\mathfrak g^*$, $W=\mathfrak g_{\le 0}^*$, and $w_0$ any nonzero vector in $\mathfrak s^*$. It follows that $S(\fg)\to S(\fg_{\le 0})$ restricts to a monomorphism $S(\fg)^{\fg}\to S(\fg_{\le 0})$.
Moreover, since $S(\fg)\to S(\fg_{\le 0})$ is $\fg_0$-equivariant, the restriction to $S(\fg)^{\fg}\to S(\fg_{\le 0})$ takes values in $S(\fg_{\le0})^{\fg_0}$. Since $\fg_0$ contains the degree operator $\xi\partial_{\xi}$, the space $S(\fg_{\le0})^{\fg_0}$ is actually contained in $S(\fg_{0})$. In particular, the composite $S(\fg)\to S(\fg_0)$ from the diagram restricts to a monomorphism $S(\fg)^{\fg}\hookrightarrow S(\fg_0)^{\fg_0}$. By the classical Harish-Chandra isomorphism, or again \cite[Proposition~1.1]{Se2}, $S(\fg_0)\to S(\fh)$ also restricts to a monomorphism on $S(\fg_0)^{\fg_0}$. In conclusion, $S(\fg)^{\fg}\to S(\fh)$ is a monomorphism.

The left vertical isomorphism restricts to an isomorphism $S(\fg)^{\fg}\cong \mathcal{Z}(\fg)$. Therefore, by the above paragraph, all we need to prove is that, with $h:U(\fg)\to S(\fh)$ the Harish-Chandra homomorphism with kernel $\left(\fg_{-1}+\fn^\fs_-\right)U(\fg)+U(\fg)\left(\fn^\fs+\fg_{1}\right)$, we have that $h(z)$ is a scalar, for every central element $z\in U(\fg)$. If we interpret $h(z)\in S(\fh)$ as a polynomial function on $\fh^\ast$, then $h(z)(\lambda)$ is the scalar by which $z$ acts on $\Delta'(\lambda)$.

For each $z\in \mathcal{Z}(\fg)$ the function $\lambda\mapsto h(z)(\lambda)$ must be constant on the set of weights $\lambda$ for which $\Delta'(\lambda)$ is in the block of $\cO$ containing the trivial module. Corollary~\ref{cor:blocks:O} and Proposition~\ref{prop:block:sl2} then imply that $h(z)$ is constant on
$$\mZ\Phi=\mZ\Phi(\fs)\oplus\mZ\delta\subset (\fh^{\fs})^\ast\oplus \mC\delta=\fh^\ast.$$
The only polynomials which are constant on this root lattice (which is of maximal rank in $\fh^\ast$ and therefore Zariski dense) are scalars, which concludes the proof.
\end{proof}
\subsection{Some invariant theory}

In this subsection, we set $\fs=\mathfrak{sl}(n)$ for some integer $n\ge 2$.

\subsubsection{}\label{Matrixg}We set $V:=L(\omega_1)$, with $\omega_1$ the first fundamental weight of $\mathfrak{sl}(n)$. By Remark~\ref{RemSimp}, we have $V\cong\mC^{n|n}$ and the representation $\fg\to\mathfrak{gl}(n|n)$ allows to realise $\fg$ as
$$\fg\;=\;\left\{ \left( \begin{array}{cc} A & b I_n\\
C & A +d I_n\\
\end{array} \right) | ~ A,C\in \mC^{n\times n},~\text{tr$A$=0=tr$C$, $b,d\in \mC$} \right\},$$
with $I_n$ the identity $n\times n$ matrix.

\subsubsection{}
In this section we investigate the invariant endomorphisms of $\otimes^rV$.
Let $S_r$ denote the symmetric group on $r$ symbols. We have a canonical algebra morphism
$$\phi_r:\;\mC S_r\to \End_{\fg}(\otimes^r V),$$
where the action of the symmetric group is given by permuting the tensor factors with appropriate minus signs, see, e.g.,  \cite[\S1]{Berele}.

\begin{theorem}\label{ThmIT}
\begin{enumerate}
\item For $r< n$ and $r=n\ge 3$, $\phi_r$ yields an isomorphism
$$\mC S_r\;\stackrel{\sim}{\to}\;\End_{\fg}(\otimes^r V).$$
For $r=n=2$, the morphism $\phi_r$ is not surjective.
\item For $r>2n-2$, the morphism $\phi_r$ is not surjective.
\item For $r>1$, the morphism $$U(\fg)\;\to\; \End_{\mC S_r}(\otimes^r V)$$
is not surjective.
\end{enumerate}
\end{theorem}
The remainder of this section is devoted to the proof of Theorem \ref{ThmIT}. We will rely on the established invariant theory for $\mathfrak{gl}(V)=\mathfrak{gl}(n|n)$ and our results on the representation theory of $\fg$.

\subsubsection{}We introduce some notation for the Lie superalgebra $\tilde{\mathfrak{g}}:=\mathfrak{gl}(n|n)$, which we realise as $2n\times 2n$ complex matrices. The elementary matrix with $1$ at the $(i,j)$th place and $0$ elsewhere is denoted by $E_{ij}$, for $1\le i,j\le 2n$. For example, with $1\le i,j\le n$, we have $E_{i+n,j}\in\fg$ if and only if $i\not=j$.  We have the standard Cartan and Borel subalgebra and furthermore we let $\{\delta_i|1\le i\le 2n\}$ be the basis dual to $\{E_{ii}|1\le i\le 2n\}$, the standard basis of the standard Cartan subalgebra.

\subsubsection{}\label{SecSW}Let $\la=(\la_1,\la_2,\ldots)$ be a partition with $\la_{n+1}\le n$. Then the transpose $\nu'$ of the partition $\nu:=(\la_{n+1},\la_{n+2},\ldots)$ is a partition of length $\ell(\nu')\le n$. We set $\la^\natural$ to be the bi-partition $(\la_1,\ldots,\la_n;\nu'_1,\ldots,\nu'_n)$, which we can interpret as a weight for $\tilde{\mathfrak{g}}$ as follows:
\begin{align}\label{weight:conv}
\la^\natural:=\sum_{i=1}^n\la_i\delta_i+\sum_{j=1}^n\nu'_j\delta_{n+j}.
\end{align}
Note that if $\ell(\la)\le n$, then $\la=\la^\natural$.

As a $U(\tilde{\mathfrak{g}})\otimes\mC S_r$-module, we have the following analogue of the classical Schur-Weyl duality \cite{Se, Berele}:
\begin{equation}\label{eqSW}\otimes^rV\;\cong\;\bigoplus_{\lambda\vdash r,\la_{n+1}\le n}\mL(\lambda^\natural)\boxtimes S_\lambda.\end{equation}
Here $S_\lambda$ stands for the simple Specht module of $S_r$ corresponding to the partition $\lambda$, while $\mL(\lambda^\natural)$ is a simple $\mathfrak{gl}(n|n)$-module of highest weight $\la^\natural$ with respect to the standard Borel subalgebra.

\begin{lemma}\label{LemCompSchur}
Consider a partition $\lambda$ with $\ell(\lambda)\le n$, so that $\la=\la^\natural$. We can interpret $\lambda$ canonically as an $\mathfrak{sl}(n)$-weight as well. Then there exists a (unique up to scalar multiple) $\fg$-morphism $p:\Delta(\lambda)\to\mL(\lambda)$. Furthermore, $p$ is an epimorphism if $\ell(\lambda)<n$.
\end{lemma}
\begin{proof}
We consider the canonical three term $\mZ$-grading on $\ftg$ compatible with the one on $\fg$.
Now define the subalgebra
$$\fa\;=\;\left\{ \left( \begin{array}{cc} aI & 0\\
0 & B\\
\end{array} \right) | ~ B\in \mC^{n\times n},~\text{tr$B$=0, $a\in \mC$} \right\}\;\subset\;\ftg.$$
Then, for any vector $F\in\ftg_{-1}\backslash \fg_{-1}$ we have
$$\ftg_0=\fg_0\oplus \fa,\quad\mbox{and}\quad \ftg_{-1}=\fg_{-1}\oplus \mC F.$$

The highest weight (with respect to $\fb$) in the $\fg$-module $\mL(\la)$ is precisely $\lambda$, from which we get the unique morphism $p$. It maps a generating highest weight vector $v_\lambda\in\Delta(\lambda)$ to a vector $w_\lambda\in\mL(\la)$. Since $\mL(\la)$ is a simple $\ftg$-module, by the PBW theorem we have
$$\mL(\la)\;=\; U(\fg_{\le0})U(\mC F)U(\fa)w_\lambda.$$
We have $\fa w_\lambda=\mC w_\lambda$ and we can now interpret $p$ as
$$U(\fg_{\le0})v_\lambda\;\to\; U(\fg_{\le0})w_\lambda + U(\fg_{\le0})Fw_\lambda,\quad uv_\lambda\mapsto uw_\lambda. $$
Hence $p$ is an epimorphism if and only if $Fw_\lambda$ is actually contained in $U(\fg_{\le 0})w_\lambda$.

Now assume $\ell(\lambda)<n$. For $i\ge n$, the facts that $E_{i,i+1}E_{i+1,i}w_\lambda=0$, that $\mL(\la)$ is a simple $\ftg$-module and that $E_{i+1,i}$ is a simple root vector imply that $E_{i+1,i}w_\lambda=0$.
Since $E_{2n,n}$ is contained in the subalgebra generated by $E_{i+1,i}$ with $i\ge n$ it follows that $E_{2n,n}w_\lambda=0$. Hence we can take $F=E_{2n,n}$.
\end{proof}

The above lemma is strict in the sense that for $\ell(\lambda)=n$ there are examples where $p$ is not surjective.
For instance, if $\lambda_n\ge n$ (and $\lambda_{n+1}=0$), we know that $\mL(\la)$ is a typical $\mathfrak{gl}(n|n)$-module and it follows that
$$\dim\Delta(\lambda)<\dim \mL(\la).$$
The lemma below shows that already the `smallest' partition with $\ell(\lambda)=n$ yields a counterexample.

 In the sequel we let $\theta$ denote the highest root in $\fs$ so that $L^0_\theta\cong\fs$.

\begin{lemma}\label{lem:not:onto}
Suppose that $\lambda=(1^n)$. As a $\fg$-module we have a non-split exact sequence
\begin{align}\label{aux:ex:seq}
0\longrightarrow M\longrightarrow \mL(\la)\longrightarrow L(-\delta)\longrightarrow 0,
\end{align}
where $M$ is a homomorphic image of $\Delta(\la)=\Delta(0)$.  Here $M$ is the radical of $\mL(1^n)$, so $\mL(1^n)$ has simple top $L(-\delta)$.  Furthermore, we have $[\mL(1^n):L(\mu)]=0$, for $\mu\not=0$ and
we have $[M:L(-\delta)]=0$.
\end{lemma}

\begin{proof}
Recall $\ftg=\mathfrak{gl}(n|n)$ with the usual gradation $\ftg=\ftg_{-1}\oplus\ftg_0\oplus\ftg_{+1}$ compatible with the $\mZ$-gradation of $\fg$. Let $\tilde{E}=\sum_{i=1}^nE_{n+i,i}\in\ftg_{-1}$ so that we have
\begin{align*}
\ftg_{-1}=\fg_{-1}\oplus\mC \tilde{E}.
\end{align*}

When $\lambda=(1^n)$ we have  the following decomposition of $\mL(\la)$ into its $-\xi\partial_\xi$-eigenspaces:
\begin{align}
\mL(1^n)\cong\wedge^n(V)=\bigoplus_{k=0}^n\mL(1^n)_{-k\delta}\quad\mbox{with}\quad \mL(1^n)_{-k\delta}=\wedge^{n-k}(V_{\oa})\otimes S^k(V_{\ob}),
\end{align}
from which we see that {$[\mL(1^n):L(\mu)]=0$, for $\mu\not=0$}. Also, we have that $\mL(\la)_{-\delta}$, as a $\fg_0$-module, is isomorphic to $\wedge^{n-1}V_{\oa}\otimes V_{\ob}$, which decomposes into $L^0(\theta-\delta)\oplus L^0(-\delta)$.

Take a $\ftg$-highest weight vector $v_\la$ of $\mL(\la)$. By irreducibility of $\mL(\la)$ we must have $\mL(\la)_{-\delta}=\ftg_{-1} v_\la$. It follows therefore that we must have
\begin{align*}
\fg_{-1}v_\la=L^0(\theta-\delta),\ \text{and }\mC\tilde{E}v_\la=L^0(-\delta).
\end{align*}
The following are easy to verify using the explicit realisation of $\ftg$ as first order differential operators on the super Grassmann algebra $\wedge^nV=\mL(\la)$:
\begin{itemize}
\item[(i)] $\fg_{-1}L^0(\theta-\delta)=\mL(\la)_{-2\delta}$.
\item[(ii)] $\partial_\xi L^0(-\delta)=\mC v_\la$.
\end{itemize}
By (i) we have $\fg_{-1}\tilde{E}v_\la\subseteq\mL(\la)_{-2\delta}=\fg_{-1}L^0(\theta-\delta)\subseteq\wedge^2(\fg_{-1})v_\la$. Hence, for $k\ge 2$, we have
$$\mL(\la)_{-k\delta}=\wedge^{k}(\ftg_{-1})v_\la= \wedge^{k}(\fg_{-1})v_\la+\wedge^{k-1}(\fg_{-1})\tilde{E}v_\la=\wedge^k(\fg_{-1})v_\la.$$

It follows now that the $U(\fg)$-module generated by $v_\la$ has codimension one and is a quotient of $\Delta(\la)$. Furthermore, $\mL(\la)/U(\fg)v_\la\cong L(-\delta)$. By (ii), the exact sequence \eqref{aux:ex:seq} is non-split and $U(\fg)v_{-\delta}=\mL(1^n)$.
\end{proof}

\begin{lemma}\label{lem:socle}
Let $\la$ be a partition with $\la_{n+1}\le n$. Denote by $\la'$ the conjugate partition of $\la$ and let $m=\sum_{i=1}^n\la'_i$ be the sum of the first $n$ parts of $\la'$. Suppose that $L(\mu-k\delta)$ lies in the socle of the $\fg$-module $\mL(\la^\natural)$. Then we have:
\begin{itemize}
\item[(i)] If $\mu=0$, then $k=m-1,m$.
\item[(ii)] If $\mu\not=0$, then $k=m-2,m-1$.
\end{itemize}
\end{lemma}

\begin{proof}
First, we note that from the character formula of $\mL(\la^\natural)$ in \cite[Section 6]{Berele}, we have the following decomposition of $\mL(\la^\natural)$ into its $-\xi\partial_\xi$-eigenspaces:
\begin{align*}
{\mL(\la^\natural)}=\bigoplus_{k=0}^{m}{\mL(\la^\natural)}_{-k\delta}.
\end{align*}

Suppose that $L(\mu-k\delta)$ is in the socle of $\mL(\la^\natural)$.

(i).  Here we suppose that $\mu=0$. Then $\fg_{-1}L(0-k\delta)=0$. There are two possibilities. Either $\tilde{E}L(0-k\delta)=0$ or else $\tilde{E}L(0-k\delta)\not=0$. In the first case, consider the $\ftg_0$-module $M=U(\ftg_0)L(0-k\delta)$, which lies in $\mL(\la^\natural)_{-k\delta}$. Since we have $\ftg_{-1}M=0$, it follows that it must contain a highest weight vector with respect to the Borel subalgebra $\ftg_{-1}+\tilde{\fb}$, where $\tilde{\fb}$ is the standard Borel subalgebra of $\ftg_0$. Now, it follows from \cite[Theorem 2.55]{CW} that such a vector has to lie in $\mL(\la^\natural)_{-m\delta}$. Thus, in this case we have $k=m$.

Now, if $\tilde{E}L(0-k\delta)\not=0$, then $\tilde{E}\tilde{E}L(0-k\delta)=0$ and also $\fg_{-1}\tilde{E}L(0-k\delta)=0$. Thus, the $U(\ftg_0)$-module $M=U(\ftg_0)\tilde{E}L(0-k\delta)$ is annihilated by $\ftg_{-1}$ and hence it contains a $\ftg$-highest weight vector with respect to the Borel subalgebra $\ftg_{-1}+\tilde{\fb}$, which lies in $\mL(\la^\natural)_{-m\delta}$. Thus, we have $-k-1=-m$, and hence $k=m-1$. This proves Part (i).

(ii). Now, assume that $\mu\not=0$. In this case we must have $\fg_{-1}L(\mu-k\delta)_{-k\delta}= L(\mu-k\delta)_{(-k-1)\delta}\not=0$ and we consider the $U(\ftg_0)$-module $M=U(\ftg_0)L(\mu-k\delta)_{(-k-1)\delta}$, if $\tilde{E}L(\mu-k\delta)_{(-k-1)\delta}=0$, or else $M=U(\ftg_0)\tilde{E}L(\mu-k\delta)_{(-k-1)\delta}$, if $\tilde{E}L(\mu-k\delta)_{(-k-1)\delta}\not=0$. Now, an analogous argument as in (i) shows that in the first case $k=m-1$ and in the second case $k=m-2$,
\end{proof}

\begin{proof}[Proof of Theorem~\ref{ThmIT}]
We start by proving Part (1). Assume $r< n$. For $\lambda\vdash r$, Lemma~\ref{LemCompSchur} implies that
$\mL(\la)$ is a quotient of $\Delta(\lambda)$ in the case when $r<n$. Lemma~\ref{LemDelta} therefore implies that for $\lambda,\mu\vdash r$ we have
$$\dim\Hom_{\fg}(\mL(\la), \mL(\mu))\;=\;\delta_{\lambda\mu}.$$
The isomorphism now follows from \eqref{eqSW}.

Now, consider the case of $r=n$. In this case the only partition of $\la$ with $\ell(\la)\ge n$ is the partition $(1^n)$. As above, in order to prove Part (1) in this case it suffices to show the following:
\begin{itemize}
\item[(i)] $\dim\Hom_{\fg}(\mL(1^n), \mL(1^n))\;=\;1$.
\item[(ii)] $\dim\Hom_{\fg}(\mL(\la), \mL(1^n))\;=\;0$, for $\la\not=(1^n)$.
\item[(iii)] $\dim\Hom_{\fg}(\mL(1^n), \mL(\la))\;=\;\delta_{n,2}$, for $\la\not=(1^n)$.
\end{itemize}
By Lemma \ref{lem:not:onto}, the $\fg$-module $\mL(1^n)$ has a simple top and this top only appears once as a constituent in $\mL(1^n)$. This implies (i).

By Lemma~\ref{LemCompSchur}, $\mL(\la)$ has a simple top $L(\lambda)$. Since this simple does not appear in $\mL(1^n)$, by Lemma \ref{lem:not:onto}, (ii) follows.

By Lemma \ref{lem:not:onto}, the radical of $\mL(1^n)$ has a simple top $L(0)$. Since $[\mL(\lambda):L(0)]=0$ we find
$$\Hom_{\fg}(\mL(1^n),\mL(\lambda))\cong \Hom_{\fg}(L(-\delta),\mL(\lambda)).$$
By Lemma \ref{lem:socle} the above space is zero when $n\ge 3$. Thus, in the case of $n\ge 3$, (iii) is proved.


Now consider the case of $r=n=2$. Since $\mL(2)=S^2(V)$, it follows easily that in $K_0(\cF)$
$$\mL(2)\;=\;L(\theta)+L(-\delta)+L(-2\delta).$$
On the other hand, Proposition~\ref{PropCartan} implies that in $K_0(\cF)$
$$P(-\delta)\;=\; L(0)+L(-\delta)+L(-3\delta)+L(-4\delta)+L(\theta-\delta)+L(\theta-2\delta)$$
and hence the entire radical of $P(-\delta)$ is in the kernel of $P(-\delta)\to\mL(2)$. Thus $L(-\delta)$ lies in the socle of $\mL(2)$ proving (iii) for $n=2$. Note that an easier argument uses the information on $\Ext^1(L(-\delta),-)$ in Corollary~\ref{ExExt1} below.
This completes the proof of Part (1).

Now we prove Part (2). We have $r>2n-2$ and we set $a:=r+2-2n>0$.
First we consider the symmetric power $S^rV$. It is well known that
$$S^rV\;\cong\;\bigoplus_{i=0}^n S^{r-i}V_{\oa}\otimes \wedge^i V_{\ob},$$
and hence $[S^rV:L]=1$ with $L:=L((r-n)\omega_1+(1-n)\delta)$ and furthermore $[S^rV:L(\nu+a\delta)]=0$ for $a\le -n$. Since, by Proposition~\ref{PropCartan}, the radical of the projective cover of $L$ only has such constituents $L(\nu+a\delta)$ with $a\le -n$, we find that
$L$ appears in the socle of $S^rV$.

On the other hand, consider the partition $\mu=(a,1^{2n-2})$ and the $\mathfrak{gl}(n|n)$-module $\mL(\mu^\natural)$. Now the $\mathfrak{gl}(n|n)$-highest weight in $\mL(\mu^\natural)$ by \eqref{weight:conv} is $a\delta_1+\sum_{i=2}^n\delta_i+(n-1)\delta_{n+1}$, which restricts to the $\fg$-weight $(a+n-2)\omega_1+(1-n)\delta$.
We thus find that
$$\mL(\mu^\natural)_{k\delta}=0\mbox{ if } k>1-n\qquad\mbox{and}\qquad  L^0((r-n)\omega_1+(1-n)\delta)\;\subset\; \mL(\mu^\natural)_{(1-n)\delta}. $$
Consequently, we have $[\mL(\mu^\natural):L]\not=0$ and by Proposition~\ref{PropCartan}, there are no other simple constituents in $\mL(\mu^\natural)$ for which $L$ appears as a constituent in their projective cover.
Thus we find that $L$ appears now in the top of $\mL(\mu^\natural)$. Again, this also follows from the information on $\Ext^1(-,L)$ in Corollary~\ref{ExExt1} below.
 In conclusion, we find a $\fg$-module morphism
$$\mL(\mu^\natural)\tto L\hookrightarrow S^rV=\mL(r\delta_1)$$
which clearly is not in the image of $\phi_r$.

Now we prove Part (3). By Lemma~\ref{LemCompSchur}, $S^rV$ is a quotient of $\Delta(r\omega_1)$. In particular it has simple top $L(r\omega_1).$
If $r>1$, we have
$$\dim (S^rV)_{\oa}\;>\; \dim (S^rV_{\oa})\;=\; \dim L(r\omega_1)_{\oa},$$
where the equality follows from Remark~\ref{RemSimp}. Hence $\mL(r\delta_1)=S^rV\tto L(r\omega_1)$ is not an isomorphism, and so the map
$$U(\fg)\;\to\;\End_{\mC S_r}(\otimes ^rV)\cong \bigoplus_{\lambda\vdash r,\la_{n+1}\le n}\End_{\mC }(\mL(\lambda^\natural))\;\to\; \End_{\mC}(S^rV)$$
cannot be an epimorphism.
\end{proof}

\begin{question}
For $\fs=\mathfrak{sl}(2)$, Theorem~\ref{ThmIT} answers the question of when $\phi_r$ is surjective completely: it is surjective if and only if $r=0,1$. Note also that for arbitrary $\fs=\mathfrak{sl}(n)$ by the results of \cite{Berele, Se} recalled in \eqref{eqSW} we know that $\phi_r$ is injective if and only if $r<(n+1)^2$. For $\fs=\mathfrak{sl}(n)$ with $n>2$, the question remains whether $\phi_r$ is surjective (or equivalently an isomorphism) in the region $[n+1,2n-2]$.
\end{question}

\section{Higher extensions and Koszulity}\label{SecExt}

In this section we compute all extension groups between simple modules in $\cF$ and prove that every non-principal block in $\cF$ is Koszul.

\subsection{Extensions}
Again we let $\lambda+a\delta$ and $\mu+b\delta$ be arbitrary weights in $\Lambda^+$. Recall the $\fs$-module morphism $D^i[V]:S^i(\fs)\otimes V\to S^{i-1}(\fs)\otimes V$ from \eqref{Stuffeq}.

\begin{theorem}\label{ThmExt}
If $i\in\mN$ and $\lambda,\mu\not=0$, we have
\begin{align*}
\dim\Ext^i_{\cF}&(L(\lambda+a\delta), L(\mu+b\delta))\\
&=\; \delta_{a-b,i}\left[\ker D^i[L^0_\lambda]:L^0_\mu\right]+\delta_{a-b,i+1}\left[\coker D^{i+1}[L^0_\lambda]:L^0_\mu\right],\\
\dim\Ext^i_{\cF}&(L(\lambda+a\delta), L(b\delta))\;=\;\delta_{a-b,i+1}[S^i\fs :L^0_\lambda],\\
\dim\Ext^i_{\cF}&(L(a\delta), L(\mu+b\delta))\;=\; \delta_{a-b,i}[S^i\fs :L^0_\lambda],\\
\dim\Ext^i_{\cF}&(L(a\delta), L(b\delta))\;=\; \begin{cases} [S^{\frac{a-b+i}{2}}\fs:\mC], \mbox{ if $a-b+i$ is even and $-i\le a-b\le i$,}\\
0,\mbox{ otherwise}.
\end{cases}
\end{align*}
\end{theorem}
First we review some special cases of the result.
As an application of Lemma~\ref{LemSurj}(2), we can simplify the expression in Theorem~\ref{ThmExt} in the following special case.
\begin{example}\label{ExampKoszul}
If $L^0_\lambda$ contains no zero-weight vectors and if $\mu\not=0$ then
$$\dim\Ext^i_{\cF}(L(\lambda+a\delta), L(\mu+b\delta))\;=\; \delta_{a-b,i}\left([S^i\fs\otimes L^0_\lambda:L^0_\mu]- [S^{i-1}\fs\otimes L^0_\lambda:L^0_\mu]\right). $$
\end{example}

Application of Lemma~\ref{LemSurj}(4) and (5) yields the following explicit expression for the first extensions.
\begin{corollary}\label{ExExt1}
Let $\theta$ denote the highest root of $\fs$. If $\lambda,\mu\not=0$, we have
\begin{align*}
&\dim\Ext^1_{\cF}(L(\lambda+a\delta), L(\mu+b\delta))\;=\; \delta_{a-b,1}(\left[\fs\otimes L^0_\lambda:L^0_\mu\right]-\delta_{\lambda,\mu}),\\
&\dim\Ext^1_{\cF}(L(\lambda+a\delta), L(b\delta))\;=\;\delta_{a-b,2}\delta_{\lambda,\theta},\\
&\dim\Ext^1_{\cF}(L(a\delta), L(\la+b\delta))\;=\; \delta_{a-b,1}\delta_{\lambda,\theta},\\
&\dim\Ext^1_{\cF}(L(a\delta), L(b\delta))\;=\; \delta_{b-a,1}.
\end{align*}
\end{corollary}

Now we start the proof of Theorem~\ref{ThmExt}.
\begin{prop}\label{ExtNab}
If $i\in\mN$ and $\lambda\not=0$, we have
\begin{align*}
\dim\Ext^i_{\cF}&(L(\lambda+a\delta), \nabla(\mu+b\delta))\\
&=\; \delta_{a-b,i}\left[\ker D^i[L^0_\lambda]:L^0_\mu\right]+\delta_{a-b,i+1}\left[\coker D^{i+1}[L^0_\lambda]:L^0_\mu\right],\\
\dim\Ext^i_{\cF}&(L(a\delta), \nabla(\mu+b\delta))\;=\; \delta_{a-b,i}[S^i\fs :L^0_\lambda].
\end{align*}

\end{prop}
\begin{proof}
For an arbitrary $M\in \cF$, we have by Shapiro's lemma
$$\Ext_{\cF}^i(M,\nabla(\mu+b\delta))\;\cong\;\Ext^i_{\cF(\fg_{\le0})}(M, L^0(\mu+b\delta)).$$
Consider the complex $C_\bullet(M)$ of $\fg_0$-modules with boundary morphisms
$$S^n\fg_{-1}\otimes M\to S^{n-1}\fg_{-1}\otimes M,\quad  Y_1Y_2\cdots Y_{n}\otimes v\mapsto \sum_{i=1}^nY_1Y_2\stackrel{\hat{i}}{\cdots} Y_n\otimes Y_i v. $$
It then follows from Lemma~\ref{LemRHom} that
\begin{eqnarray*}
\Ext^i_{\cF(\fg_{\le0})}(M, L^0(\mu+b\delta))&\cong& H^i(\Hom_{\fg_0}(S^\bullet\fg_{-1}\otimes M, L^0(\mu+b\delta)))\\
&\cong& \Hom_{\fg_0}(H_i(C_\bullet(M)), L^0(\mu+b\delta)).
\end{eqnarray*}

Consider first $M=L(a\delta)$. Then the boundary morphisms in $C_\bullet(M)$ are zero, so we find
$$H_i(C_\bullet(M))\;\cong\; C_i(M)\;\cong\; S^i\fs\otimes L^0((a-i)\delta),$$
which implies
$$\dim\Ext_{\cF}^i(M,\nabla(\mu+b\delta))\;=\;[S^i\fs\otimes L^0((a-i)\delta):L^0(\mu+b\delta)]\;=\;\delta_{a-i,b}[S^i\fs:L^0_\mu].$$

Next we consider $M=L(\lambda+a\delta)$, for $\lambda\not=0$. Using Remark~\ref{RemSimp} we can describe the complex $C_\bullet(M)$ as follows. As an $\fs$-module, we have
$$C_n(M)\cong (S^n\fs\otimes L^0_\lambda)\oplus (S^n\fs\otimes \Pi L^0_\lambda),$$
and the boundary homomorphisms are
$$(Y_1 \cdots Y_n\otimes u,Z_1 \cdots Z_n\otimes \Pi w)\mapsto (0, \sum_{i=1}^n Y_1\stackrel{\hat{i}}{\cdots} Y_n\otimes Y_i \Pi u).$$
Plugging in the $\xi\partial_\xi$-action from Remark~\ref{RemSimp} thus yields
$$H_i(C_\bullet( M))\;\cong\; \left(\ker D^i[L^0_\lambda]\right) \otimes L^0((a-i)\delta)\,\oplus\,\left( \coker D^{i+1}[L^0_\lambda]\right)\otimes L^0((a-i-1)\delta),$$
which concludes the proof.
\end{proof}

\begin{proof}[Proof of Theorem~\ref{ThmExt}]
By Lemma~\ref{LemSimp} we have $L(\mu+b\delta)\cong\nabla(\mu+b\delta)$ when $\mu\not=0$, so Proposition~\ref{ExtNab} implies the first and the third equations in the theorem. The second equation follows from the third, together with Corollary~\ref{CorBGG}(2) and the fact that $\fs^\ast\cong\fs$.

By Lemma~\ref{LemRHom} we have
\begin{eqnarray*}\Ext^i_{\cF}(L(a\delta),L(b\delta))&\cong& \Hom_{\fg_{0}}(S^i\fg_{\ob}, L^0((b-a)\delta))\\
&\cong&\bigoplus_{j=0}^i \Hom_{\fg_{0}}(L^0((i-2j)\delta)\otimes S^j\fs, L^0((b-a)\delta)),
\end{eqnarray*}
which concludes the proof.
\end{proof}

\subsection{Koszulity}
We refer to \cite[\S 2.1]{MOS} for details on the notions of $\mZ$-graded categories, homogeneous functors, and categories with free $\mZ$-action. In short, a $\mZ$-graded category is one enriched in the monoidal category of $\mZ$-graded vector spaces.

\subsubsection{}We have the derived subalgebra
$$\fg':=[\fg,\fg]=\fg_{-1}\oplus\fs\oplus\fg_1,$$
which is again a classical Lie superalgebra of type I.
Because of this structure, for $i\in\mZ$, we can consider the algebra automorphism
$$\varphi_i:U(\fg)\to U(\fg),\quad \mbox{determined by } \begin{cases}
X\mapsto X,&\mbox{ for $X\in\fg'$}\\
\xi\partial_\xi\mapsto \xi\partial_\xi +i.
\end{cases}$$
This morphism induces a functor (an auto-equivalence) $\langle i\rangle:\cF\to\cF$ which twists each module by $\varphi_i$.

\subsubsection{}
Let $\cP_\circ$ denote a skeletal category equivalent to the category of indecomposable projective modules in $\cF$.
The functors $\langle i\rangle$ restrict to $\cP_\circ$ and acts as
$$P(\lambda+a\delta)\mapsto P(\lambda+a\delta)\langle i\rangle=P(\lambda+(a-i)\delta).$$ The collection of these functors yields a free $\mZ$-action on $\cP_\circ$.
We define the $\mC$-linear $\mZ$-graded category $\ba$ as the quotient $\cP_\circ/\mZ$ with respect to this free $\mZ$-action.

We describe this category explicitly. Let $\Lambda^+_0$ denote the set of integral dominant $\fs$-weights. In particular, we have a canonical bijection
$$\Lambda^+\,\to\,\Lambda^+_0\times\mZ,\;\; \lambda+a\delta\mapsto (\lambda,a).$$
Then we have $\Ob\ba=\Lambda^+_0$ and the graded morphism spaces are determined as follows. For $\lambda,\mu\in \Lambda^+_0$ and $i\in\mZ$, we set
$$\ba(\lambda,\mu)_i\;=\;\Hom_{\fg}(P(\lambda),P(\mu+i\delta))\;\cong\;\Hom_{\fg}(P(\lambda+a\delta), P(\mu+(i+a)\delta)).$$
The isomorphism is used in the definition of the composition of morphisms in $\ba$.

\subsubsection{}Let $\vecc_{\mZ}$ be the category of finite-dimensional $\mZ$-graded vector spaces.
It then follows from \cite[\S 2.1]{MOS}, that $\cF$ is equivalent to the category of homogeneous degree zero functors $\ba\to\vecc_{\mZ}$. In other words, $\cF$ is equivalent to the category $\ba\gmod$ of finite-dimensional graded modules of the graded algebra (without unit) $\bigoplus_{\lambda,\mu}\ba(\lambda,\mu)$.

\begin{remark}\label{RemKoszul}
By definition, $\cF\simeq \ba\gmod$ is the `graded lift' of the abelian category of all $\mC$-linear functors $\ba\to\vecc$, where we thus ignore the $\mZ$-grading on $\ba$. Usually, it is the latter functor category which has a Lie-theoretic origin, see \cite{BGS, BS1, BLW, CG, CL, GM, GS}, and the existence of the graded lift is part of `the theorem'. In contrast, in the present setup, our module category already acts as the graded lift itself.

Naturally, one can still wonder how to interpret the degraded category of $\cF$, that is the category of functors $\ba\to\vecc$. As it turns out, it is more elegant to answer a `super' version of the question. Rather than completely forgetting the $\mZ$-grading on $\ba$, we can restrict it to a $\mZ/2$-grading. One can then show that the category of homogeneous degree zero (for the $\mZ/2$-grading) functors $\ba\to\vecc_{\mZ_2}$ is equivalent to $\cF(\fg')$,
with $\fg'$ the derived algebra of $\fg$. The fact that we did not forget the grading completely results in a $\mZ/2$-action on the resulting module category, as a shadow of the $\mZ$-action on $\cF$. This is precisely the action of $\Pi$ on $\cF(\fg')$. \end{remark}

\subsubsection{}
It is a priori clear that $\ba(\lambda,\mu)=0$ unless $\lambda-\mu\in{ \mZ}\Phi(\fs)$.
Consequently, we can decompose the category $\ba$ as
$$\ba\;=\;\coprod_{\nu}\ba^{\nu},$$
where $\nu$ ranges over the set of minuscule $\fs$-weights (as algebras one can also write $\ba=\bigoplus_{\nu}\ba^\nu$). It follows easily from Theorem~\ref{thm:class:fd} and Proposition~\ref{prop:block:sl2} that each category $\ba^{\nu}$ is connected, so the above decomposition is exhaustive.

Furthermore, Theorem~\ref{thm:class:fd} and Proposition~\ref{prop:block:sl2} also imply that $\fa^{\nu}\gmod$ is indecomposable (as an additive category) in all cases except when $\fs=\mathfrak{sl}(2)$ and $\nu$ is the fundamental weight. In the latter case $\fa^\nu\gmod$ decomposes into two blocks, which are interchanged by the auto-equivalence $\langle 1\rangle$. In particular, the second and third block in Proposition~\ref{prop:block:sl2} are equivalent.

\begin{lemma} If $\nu\not=0$, then $\ba^{\nu}$ is positively graded in the sense of \cite[Definition~1]{MOS}, meaning for $\lambda,\mu\in\nu+\mZ\Phi(\fs)$ we have
\begin{enumerate}
\item $\ba^{\nu}(\lambda,\mu)_i=0$ for $i<0$;
\item $\ba^{\nu}(\lambda,\mu)_0=\mC \delta_{\lambda,\mu}$;
\item $\bigoplus_\kappa\ba^{\nu}(\lambda,\kappa)_i$ and $\bigoplus_\kappa\ba^{\nu}(\kappa,\mu)_i$ are finite-dimensional, for $i>0$.
\end{enumerate}
Furthermore, $\ba^0$ is not positively graded.
\end{lemma}
\begin{proof}
Proposition~\ref{PropCartan} implies easily that (3) is satisfied in every $\ba^{\nu}$, whereas (1) and (2) are satisfied if and only if $\nu\not=0$.
\end{proof}

Following \cite[\S 5.4]{MOS}, a positively graded category $\bc$ is Koszul if each simple object in $\bc\gmod$ admits a linear (in the grading) projective resolution. For a simple object $L$ in degree zero, this means that there exists a projective resolution $P_\bullet$ such that the top of $P_i$ lives in degree $i$. Clearly a linear projective resolution is a minimal projective resolution and it exists if and only if
$$\Ext^i_{\bc\gmod}(L,L'\langle j\rangle)=0\quad\mbox{whenever}\;i\not=j,$$
for every two simple modules $L,L'$ in degree zero.

\begin{prop}\label{propKoszul}
The graded category $\ba^{\nu}$ is Koszul when $\nu\not=0$.
\end{prop}
\begin{proof}
This follows immediately from Example~\ref{ExampKoszul}.
\end{proof}

\subsection{Further comments on Koszulity}

\subsubsection{} One way to formulate Proposition~\ref{propKoszul} is as follows. Every block in the category $\cF(\fg')$ admits a $\mZ$-graded lift (in the informal `super' sense of Remark~\ref{RemKoszul}). Whenever the block is not principal (does not contain a trivial $\fg'$-module), this graded lift is Koszul (more precisely the category of graded modules of a Koszul category). This leaves open the question of whether the principal block in $\cF(\fg')$ admits some different graded lift which might be Koszul.

This seems unlikely and one can easily prove for $\fs=\mathfrak{sl}(2)$ that this is not possible. Let $\omega$ be the fundamental weight and consider the standard module $\Delta(2\omega)$. When we ignore parity, it has socle filtration
$$L(2\omega),\; L(4\omega)\oplus \mC,\; \mC,\; L(2\omega),$$
and radical filtration
$$L(2\omega),\; \mC\;,L(4\omega)\oplus \mC,\; L(2\omega).$$
However, it is well known, see \cite[Proposition~2.4.1]{BGS}, that modules with simple socle and simple top must be rigid (the socle and radical filtrations must coincide) over any Koszul category.

\subsubsection{Connection with Conformal Modules} The first extension groups in $\cF$ described in Corollary \ref{ExExt1} have remarkable similarities to the first extension groups between so-called finite conformal modules over the current conformal algebra computed in \cite{CKW}. Below, we shall not give the precise definition of conformal modules but instead define the corresponding equivalent category of modules over the so-called extended annihilation algebra, which will denoted by $\mathcal G$ below. The interested reader is referred to \cite{CK} for further detail.

Let $\fs$ be a finite-dimensional simple Lie algebra as before. Let $\mC[t]$ be the polynomial algebra so that we can form the Lie algebra $\fs\otimes\mC[t]$. Set $\partial_t:=\frac{\partial}{\partial t}$ and let $\mathcal D=\mC\partial_t+\mC t\partial_t$. We consider the (infinite-dimensional) semisimple extension
\begin{align*}
\mathcal G:=\fs\otimes\mC[t]\rtimes\mathcal D.
\end{align*}
Our Lie superalgebra \eqref{Defg} can be regarded as the `super analogue' of $\mathcal{G}$.
The Lie algebra $\mathcal G$ is $\mZ$-graded with grading operator $t\partial_t$, i.e., $\mathcal G=\bigoplus_{j=-1}^\infty\mathcal G_j$ with $\mathcal G_{-1}=\mC \partial_t$, $\mathcal G_0=\fs \oplus\mC t\partial_t$, and $\mathcal G_j=\fs\otimes t^j$, for $j>0$. Set $\mathcal G_{\ge 0}:=\bigoplus_{j\ge 0}\mathcal G_j$.

 Let $\mathcal C$ be the category of $\mathcal G$-modules $M$ such that $t\partial_t$ acts semisimply with integer eigenvalues and $M$ is finitely generated over $\mC[\partial _t]$. It follows readily that for each vector $v\in M$ there exists a positive integer $N$ (depending on $v$) with $\mathcal G_j v=0$, for all $j\ge N$.

 Let $\la\in(\fh^\fs)^*$ be a dominant integral $\fs$-weight and define $\delta\in\fh^\ast$ by $\delta(-t\partial_t)=1$ extended trivially to $\fh$.
We have the finite-dimensional irreducible $\mathcal G_0$-module $L^0(\la+a\delta)$ of highest weight $\la+a\delta$, $a\in\mZ$. We can extend the action of $\mathcal G_0$ on $L^0(\la+a\delta)$ to an action of $\mathcal G_{\ge 0}$ in a trivial way. Then we can construct the induced module
\begin{align*}
\text{ind}_{{\mathcal G}_{\ge 0}}^{\mathcal G} L^0(\la+a\delta).
\end{align*}
If $\la\not=0$, it turns out that $\text{ind}_{{\mathcal G}_{\ge 0}}^{\mathcal G} L^0(\la+a\delta)$ is irreducible. On the other hand, if $\la=0$, then $L^0(a\delta)$ extends trivially to an irreducible $\mathcal G$-module.

We summarise the above discussion in the following.

\begin{prop}\cite[Corollary 3.1]{CK}
The simple objects in $\mathcal C$ are parameterized by pairs $(\la,a)$, where $\la$ is a dominant integral weight of $\fs$ and $a\in\mZ$. Denoting the corresponding simple  modules by $\tL(\la+a\delta)$ we have
\begin{itemize}
  \item[(i)] $\tL (\la+a\delta)=\text{ind}_{{\mathcal G}_{\ge 0}}^{\mathcal G} L^0(\la+a\delta)$, if $\la\not=0$.
  \item[(ii)] $\tL (a\delta)=L^0(a\delta)$ extended trivially, if $\la=0$.
\end{itemize}
\end{prop}

The first extension groups between the simple objects in $\mathcal C$ have been classified in \cite[Section 4]{CKW}.

\begin{theorem}\label{ThmConf} Let $\theta$ denote the highest root of $\fs$. If $\lambda,\mu\not=0$, we have
\begin{align*}
&\dim\Ext^1_{\cC}(\tL (\lambda+a\delta), \tL (\mu+b\delta))\;=\; (\delta_{a-b,1}+\delta_{a-b,2}\delta_{\fs,\mathfrak{sl}(2)})(\left[\fs\otimes L^0_\lambda:L^0_\mu\right]-\delta_{\lambda,\mu}),\\
&\dim\Ext^1_{\cC}(\tL (\lambda+a\delta), \tL (b\delta))\;=\;\delta_{a-b,1}\delta_{\lambda,\theta},\\
&\dim\Ext^1_{\cC}(\tL (a\delta), \tL (\la+b\delta))\;=\; 0,\\
&\dim\Ext^1_{\cC}(\tL (a\delta), \tL (b\delta))\;=\; \delta_{b-a,1}.
\end{align*}
\end{theorem}

Henceforth we exclude $\fs=\mathfrak{sl}(2)$. Then we can draw the following conclusions from Theorems~\ref{ThmConf} and Corollary~\ref{ExExt1}.
\begin{itemize}
\item There is a canonical correspondence between the blocks in the category $\cC$ and the blocks in $\cF$. Indeed, while the generators of the equivalence relations coming from the $\Ext^1$-quivers are not exactly the same, it is easy to see they generate the same equivalence relations in the sense that $\tL(\lambda+a\delta)$ will be in the same block of $\cC$ as $\tL(\mu+b\delta)$ if and only if $L(\lambda+a\delta)$ is in the same block of $\cF$ as $L(\mu+b\delta)$.
\item Take a block in $\cF$ which is Koszul (so any block but the principal one). The canonical identification of labels yields an isomorphism of the $\Ext^1$-quiver of simple objects with the one for the corresponding block in $\cC$.
\end{itemize}

The above observations might raise hopes that one could realise the Koszul dual of blocks in $\cF$ as appropriate module categories of $\mathcal{G}$. However, it is not clear that this might be possible. The most naive guesses fail to be the Koszul dual, since $\mathcal{G}$ seems to be `too big', leading to graded morphism spaces between projective objects which are bigger than the corresponding extension spaces in Example~\ref{ExampKoszul}. We conclude by looking at a similar situation.
\begin{remark}
For a simple finite-dimensional $\fs$-representation $V$, we can construct the generalised Takiff algebra $\fs\ltimes V$ and superalgebra $\fs\ltimes \Pi V$ similarly as described in the Introduction.
In \cite[\S 5.1]{GM}, see also~\cite{CG}, it is observed that a category of $\fs\ltimes V$-modules admits a graded lift which is Koszul and moreover it is Koszul dual to a graded lift of a category of $\fs\ltimes \Pi V$-modules. We can reformulate this for the special case where $V$ is the adjoint representation as follows.
For the Lie superalgebra $\fk:=(\fs\otimes\wedge(\xi))\rtimes \mC\xi\partial_\xi$, the category $\cF(\fk)$ is Koszul and the Koszul dual is given by a category of modules over $(\fs\otimes\mC[t]/(t^2))\rtimes \mC t\partial_t$. It is easy to see that the truncation $t^2=0$ is essential for the Koszul duality. In our setting we have $\partial_t\in\mathcal{G}$, so such a truncation to make $\mathcal{G}$ smaller is not possible.
\end{remark}


\subsection*{Acknowledgement}
The first author is supported by a MoST grant of the R.O.C. The second author is supported by the ARC grant DE170100623. The authors thank Arun Shrinath Kannan and Honglin Zhu for pointing out that the rank 2 case of the semisimple extension of the Takiff superalgebra yields a periplectic superalgebra.

\end{document}